\documentclass[reqno]{amsart}
\usepackage[margin=3cm]{geometry}
\usepackage{amsmath}
\usepackage{pinlabel}
\usepackage{bookmark}
\hypersetup{colorlinks=true,citecolor=black,urlcolor=blue,linkcolor=black}
\usepackage{amssymb}
\usepackage{extarrows}
\usepackage{mathabx}
\usepackage[all]{xy}
\usepackage{amsbsy}
\usepackage{colonequals}
\usepackage{graphicx}
\usepackage{appendix}
\usepackage{float}
\usepackage{subfigure}
\usepackage[numbers,sort&compress]{natbib}
\usepackage{amsthm}
\usepackage{geometry}
\usepackage{fancyhdr}
\usepackage{color} 
\usepackage{overpic}
\usepackage{enumerate} 
\usepackage{mathrsfs}
\usepackage{colonequals}
\usepackage{microtype}
\usepackage{hyperref}
\usepackage{cleveref}
\usepackage{diagbox}
\usepackage[thinlines]{easytable}
\usepackage{comment}
\usepackage{tikz}
\usepackage{tikz-cd}

\newtheorem{thm}{Theorem}[section]
\newtheorem{cor}[thm]{Corollary}
\newtheorem{prop}[thm]{Proposition}
\newtheorem{lem}[thm]{Lemma}

\newtheorem{proposition}[thm]{Proposition}
\newtheorem{corollary}[thm]{Corollary}

\theoremstyle{definition}
\newtheorem{defn}[thm]{Definition}
\newtheorem{cons}[thm]{Construction}
\newtheorem{exmp}[thm]{Example}
\newtheorem{conj}[thm]{Conjecture}
\newtheorem*{fact}{Fact}

\newtheorem*{org}{Organization}
\newtheorem*{ack}{Acknowledgement}

\theoremstyle{remark}
\newtheorem{rem}[thm]{Remark}

\numberwithin{equation}{section}

\newtheorem{quest}[thm]{Question}

\newcommand{\beq}{\begin{equation*}\begin{aligned}}
\newcommand{\eeq}{\end{aligned}\end{equation*}}
\newcommand{\bpf}{\begin{proof}}
\newcommand{\epf}{\end{proof}}
\newcommand{\bthm}{\begin{thm}}
\newcommand{\ethm}{\end{thm}}
\newcommand{\bprop}{\begin{prop}}
\newcommand{\eprop}{\end{prop}}
\newcommand{\bcor}{\begin{cor}}
\newcommand{\ecor}{\end{cor}}
\newcommand{\blem}{\begin{lem}}
\newcommand{\elem}{\end{lem}}
\newcommand{\bdefn}{\begin{defn}}
\newcommand{\edefn}{\end{defn}}
\newcommand{\bcons}{\begin{cons}}
\newcommand{\econs}{\end{cons}}
\newcommand{\bexmp}{\begin{exmp}}
\newcommand{\eexmp}{\end{exmp}}
\newcommand{\brem}{\begin{rem}}
\newcommand{\erem}{\end{rem}}
\newcommand{\bfa}{\begin{fact}}
\newcommand{\efa}{\end{fact}}
\newcommand{\benu}{\begin{enumerate}[(1)]}
\newcommand{\eenu}{\end{enumerate}}

\newcommand{\bbslash}{\backslash\backslash}

\newcommand{\al}{\alpha}
\newcommand{\be}{\beta}
\newcommand{\ga}{\gamma}
\newcommand{\Ga}{\Gamma}

\newcommand{\id}{\operatorname{Id}}
\newcommand{\ish}{I^\sharp}

\newcommand{\ina}{I^\natural}

\newcommand{\mcm}{\mathcal{M}}

\newcommand{\p}{\prime}
\newcommand{\pp}{{\prime\prime}}

\newcommand{\aand}{~{\rm and}~}

\newcommand{\Z}{\mathbb{Z}}
\newcommand{\F}{\mathbb{F}}

\newcommand{\R}{\mathbb{R}}
\newcommand{\C}{\mathbb{C}}
\newcommand{\intg}{\mathbb{Z}}
\newcommand{\ft}{{\mathbb{F}_2}}

\newcommand{\posi}{\mathbb{N}_+}
\newcommand{\comp}{\mathbb{C}}

\newcommand{\xra}{\xrightarrow}
\newcommand{\xla}{\xleftarrow}

\DeclareMathOperator{\im}{Im}

\DeclareMathOperator{\cone}{Cone}
\DeclareMathOperator{\ind}{Ind}
\def\<#1>{\mathinner{\langle#1\rangle}}
\def\|#1|{\mathinner{\lVert #1 \rVert}}

\newcommand{\deq}{\colonequals}
\newcommand{\ybnk}{Y\backslash\backslash K}
\newcommand{\pybnk}{\partial(Y\backslash\backslash K)}
\newcommand{\cstar}{\mathbb{C}^\times}
\newcommand{\rk}{\operatorname{rank}}
\newcommand{\pt}{\{\operatorname{pt}\}}

\newcommand{\rp}[1][3]{\mathbb{RP}^{#1}}

\newcommand{\ad}{\mathrm{ad}}

\definecolor{lygreen}{HTML}{016646}
\definecolor{cpurple}{HTML}{800080}
\begin{document}

\title{Instanton \texorpdfstring{$2$}{2}-torsion and fibered knots}

\author{Deeparaj Bhat}
\address{Department of Mathematics, Columbia University}
\curraddr{}
\email{d.bhat@columbia.edu}
\thanks{}

\author{Zhenkun Li}
\address{Academy of Mathematics and Systems Science, Chinese Academy of Sciences}
\curraddr{}
\email{zhenkun@amss.ac.cn}
\thanks{}

\author{Fan Ye}
\address{Department of Mathematics, Harvard University}
\curraddr{}
\email{fanye@math.harvard.edu}
\thanks{}

\date{}
\dedicatory{}
\begin{abstract}
We prove that the unreduced singular instanton homology $I^\sharp(Y,K;\mathbb{Z})$ has $2$-torsion for any null-homologous fibered knot $K$ of genus $g>0$ in a closed $3$-manifold $Y$ except for $\#^{2g}S^1\times S^2$. The main technical result is a formula of $I^\sharp(Y,K;\mathbb{C})$ via sutured instanton theory, by which we can compare the dimensions of $I^\sharp(Y,K;\mathbb{F}_2)$ and $I^\sharp(Y,K;\mathbb{C})$. As a byproduct, we show that $I^\sharp(S^3,K;\mathbb{C})$ for a knot $K\subset S^3$ admitting lens space surgeries is determined by the Alexander polynomial, while some special cases of torus knots have been previously studied by many people. Another byproduct is that the next-to-top Alexander grading summand of instanton knot homology $KHI(S^3,K,g(K)-1)$ is non-vanishing when $K$ has unknotting number one, which generalizes the Baldwin--Sivek's result in the fibered case. Finally, we discuss the relation to the Heegaard Floer theory.
\end{abstract}
\maketitle


\section{Introduction}

Singular instanton homology was introduced by Kronheimer--Mrowka \cite{kronheimer2011khovanov} to prove that Khovanov homology detects the unknot. In this paper, we focus on the unreduced variant $\ish(Y,K;\Z)$, which is a $\mathbb{Z}$-module constructed for any knot $K$ in a (closed, oriented, connected) $3$-manifold $Y$. From \cite[Theorem 8.2]{kronheimer2011khovanov}, when $Y=S^3$, there is a spectral sequence from the (unreduced) Khovanov homology $Kh(\widebar{K};\Z)$ of the mirror knot $\widebar{K}$ to $\ish(S^3,K;\Z)$. 

To the authors' knowledge, all calculated examples of Khovanov homology have $2$-torsion. Indeed, Shumakovitch \cite[Conjecture 1]{Shu14torsion} conjectured that $2$-torsion exists in $Kh(K;\Z)$ for any nontrivial knot $K$. Gujral--Wang \cite{GW2025Kh} provides a minimality property for knots without Khovanov $2$-torsion. In particular, knots with unknotting number one have Khovanov $2$-torsion. 

We ask a similar question for $\ish(Y,K;\Z)$ as follows.

\begin{quest}\label{qu: 2torsion}
For any nontrivial knot $K$ in $S^3$ or a general $3$-manifold $Y$, does $2$-torsion always exist in $\ish(Y,K;\Z)$?
\end{quest}
The construction of $\ish(Y,K)$ also provides clues for this question. Roughly speaking, the homology $\ish(Y,K)$ can be regarded as a variant of homology of the traceless representation variety\[R_0(Y,K)=\{\rho:\pi_1(Y\backslash K)\to SU(2),~\operatorname{tr}(\rho(\mu))=0 \text{ for the meridian } \mu \text{ of }K\},\]deformed by differentials from solutions of Yang-Mills equations called \emph{instantons}. When $\rho$ has nonabelian image (called an \emph{irreducible representation}), the conjugation action on $\rho$ generates a copy of $SU(2)/Z(SU(2))\cong SO(3)\cong \mathbb{RP}^3$ inside $R_0(Y,K)$. Note that\[H_*(SO(3);\Z)\cong \intg\oplus \intg\oplus\intg/2.\]By the detection result in \cite{kronheimer2011khovanov}, at least for nontrivial knots in $S^3$, such a copy always exists. In particular, when $K$ is the torus knot $T(2,n)$, we know from \cite[Observation 1.1]{kronheimer2011knot} that $R_0(S^3,K)$ consists of a copy of $S^2$ and several copies of $SO(3)$, and \[Kh(\widebar{K};\Z)\cong H_*(R(S^3,K);\Z)\cong \ish(S^3,K;\Z).\]

Moreover, it is well-known (cf.\ \cite[Theorem B.2]{LY2025torsion} for a proof based on previous references) that for any quasi-alternating knot, the spectral sequence from $Kh(\widebar{K};\Z)$ to $\ish(S^3,K;\Z)$ collapses and they both have $2$-torsion.

In this paper, we answer Question \ref{qu: 2torsion} for fibered knots in the following theorem.
\bthm\label{thm: knot has 2 torsion}
Suppose $K$ is a null-homologous knot of genus $g>0$ in $Y\not\cong \#^{2g}S^1\times S^2$. If $K$ is fibered, then $\ish(Y,K;\Z)$ has $2$-torsion.
\ethm
\brem
Motivated by Question \ref{qu: 2torsion}, the last two authors \cite[Propositions 1.5-1.7]{LY2025torsion} also showed the existence of $2$-torsion in $\ish(Y,K;\Z)$ in the following three cases.
\begin{itemize}
    \item $Y=S^3$ and $K$ has Seifert genus $g(K)=1$ and the Alexander polynomial $\Delta_K(t)\neq 1$;
    \item $Y=S^3$ and $K$ has unknotting number one;
    \item $Y=S^3_1(J)$ is the $1$-surgery on a nontrivial knot $J\subset S^3$ and $K=\widetilde{J}_1$ is the dual knot.
\end{itemize}

On the other hand, when the knot is empty, the $\Z$-module $\ish(Y;\Z)=\ish(Y,\emptyset;\Z)$ is called the \emph{framed instanton homology} and was introduced by Kronheimer--Mrowka \cite{kronheimer2010knots,kronheimer2011knot,kronheimer2011khovanov}. See \cite[\S 2.1]{LY2025torsion} for a quick review. One can also ask about the existence of $2$-torsion in $\ish(Y;\Z)$. Some examples were studied by the authors of this paper \cite{bhat2023newtriangle,LY2025torsion,LY20255surgery,LY2025dimension,Ye2025dualknot} and Ghosh--Miller-Eismeier \cite{SGMME}.
\erem






\subsection{Sketch of the proof}\label{sec: Sketch of the proofs}
From the universal coefficient theorem, the homology of a free and finitely generated chain complex has $2$-torsion if and only if the dimension in $\F_2$ coefficients is strictly larger than the dimension in $\comp$ coefficients. By \cite[Corollary 8.7]{daemi2019equivariant} (see also \cite[Lemma 7.7]{kronheimer2019web1} and \cite[Proposition 4.4]{Xie2021earring}), we have \begin{equation}\label{eq: 2dim}
    \dim \ish(Y,K;\F_2)=2\dim \ina(Y,K;\F_2),
\end{equation}where $\ina$ is the reduced variant of singular instanton homology that is related to the reduced Khovanov homology $\widetilde{Kh}$ via a spectral sequence \cite{kronheimer2011khovanov}. By the universal coefficient theorem and \cite[Proposition 1.4]{kronheimer2011khovanov}, we have\begin{equation}\label{eq: univer}
    \dim \ina(Y,K;\F_2)\ge \dim \ina(Y,K;\comp)=\dim KHI(Y,K),
\end{equation}
where $KHI(Y,K)$ is a $\comp$-vector space obtained from sutured instanton homology in \cite[\S 7]{kronheimer2010knots}. Combining (\ref{eq: 2dim}) and (\ref{eq: univer}), to show the existence of $2$-torsion in $\ish(Y,K)$, it suffices to prove \begin{equation}\label{eq: iff}
    2\dim KHI(Y,K)>\dim \ish(Y,K;\comp).
\end{equation}

The main technical result of this paper is the following.

\bthm\label{thm: formula for singular instanton}
Suppose $K$ is a knot in a closed oriented connected 3-manifold $Y$. Then there are two endomorphisms $d_{1}^+$ and $d_{1}^-$ of $KHI(Y,K)$ constructed in Definition \ref{defn: d1 differential} and  two scalars $c_+,c_-\in \cstar$ so that there is an exact triangle
\begin{equation}\label{eq: exact trianlge formula}
\xymatrix@R6ex{
KHI(Y,K)\ar[rr]^{c_+ d_{1}^++c_- d_{1}^-}&& KHI(Y,K)\ar[dl]^{}\\
&\ish(Y,K;\comp)\ar[ul]^{}&
}	
\end{equation}
Moreover, if $K$ is rationally null-homologous, there is a $\Z$-grading on $KHI(Y,K)$ induced by the (rational) Seifert surface of $K$ (called the \emph{Alexander grading}). The maps $d_{1}^\pm$ are homogeneous with different grading shifts. In this case, we have
\[\begin{aligned}
    \ish(Y,K;\comp)\cong & H_*\left(\cone(KHI(Y,K)\xra{c_+ d_{1}^++ c_- d_{1}^-} KHI(Y,K))\right)\\\cong & H_*\left(\cone(KHI(Y,K)\xra{ d_{1}^++ d_{1}^-} KHI(Y,K))\right).
\end{aligned}\]
\ethm
\brem
Kronheimer--Mrowka \cite[Figure 13]{kronheimer2011khovanov} used the skein exact triangle to deduce the following exact triangle over any coefficients\begin{equation*}
\xymatrix@R6ex{
\ina(Y,K)\ar[rr]^{\theta}&& \ina(Y,K)\ar[dl]^{}\\
&\ish(Y,K)\ar[ul]^{}&
}	
\end{equation*}The map $\theta$ vanishes over $\F_2$, which is equivalent to (\ref{eq: 2dim}). To prove (\ref{eq: iff}) and hence the existence of $2$-torsion in $\ish(Y,K;\Z)$, it suffices to show $\theta$ is nontrivial over $\C$. However, to the authors' knowledge, the only calculated examples for $\theta$ are $2$-bridge knots by Daemi--Scaduto \cite[\S 9.2]{daemi2019equivariant}. In general, it is hard to compute $\theta$ from its definition.
\erem
\brem
The maps $d_{1}^\pm$ in Theorem \ref{thm: formula for singular instanton} are the differentials introduced in \cite[Theorem 3.20]{LY2021large} (written as $d_{1,\pm}$) on the first pages of spectral sequences from $KHI(Y,K)$ to $\ish(Y;\comp)$ related to positive and negative bypass maps, respectively, which are the source of the subscripts. Note that when $K$ is not rationally null-homologous, there is no spectral sequence from $KHI(Y,K)$ to $\ish(Y;\comp)$ (neither in the Heegaard Floer setting, e.g.\ for the knot $S^1\times \pt$ inside $S^1\times S^2$). However, we can still define the maps $d_{1}^\pm$ by contact gluing maps as in Definition \ref{defn: d1 differential}; see also \cite[\S 3.1]{LY20255surgery}.
\erem
Based on Theorem \ref{thm: formula for singular instanton}, the proof of Theorem \ref{thm: knot has 2 torsion} reduces to the fact that $d_{1}^++ d_{1}^-$ is nonzero by the fiberness assumption. Indeed, the discussion by Baldwin--Sivek \cite[\S 1.4]{BS2022khovanov} implies a stronger result that both $d_{1}^\pm$ are nonzero. Since they have different grading shifts, the sum is also nonzero, and we conclude the proof. More details will be given in \S \ref{sec: Proofs of the main theorems}.


As a byproduct application of Theorem \ref{thm: formula for singular instanton}, for some special families of knots, the maps $d_{1}^\pm$ were studied by the last two authors \cite{LY2021large,LY2022integral2,LY20255surgery} and Theorem \ref{thm: formula for singular instanton} implies the following computations of $\ish(Y,K;\comp)$. 

\bprop\label{prop: l space knot computing}
Suppose $K\subset S^3$ is an instanton L-space knot as in \cite{LY2021large} (e.g.\ a knot admitting lens space surgeries). Then $\ish(S^3,K;\comp)$ is determined by the Alexander polynomial $\Delta_K(t)$. 
\eprop
\bpf
For an instanton L-space knot $K$, \cite[Theorem 1.9]{LY2021large} implies\[\Delta_K(t)=t^{n_k}-t^{n_{k-1}}+t^{n_{k-2}}-\cdots+(-1)^k+\cdots+t^{-n_{k-2}}-t^{-n_{k-1}}+t^{-n_k}\]for integers $n_k>n_{k-1}>\cdots>n_0=0$. Let $m_k=n_k-n_{k-1}$. By \cite[Theorem 5.11]{LY2021large}, up to the mirror knot, the differentials $d_{r}^\pm$ on $KHI(S^3,K)$ (written as $d_{r,\pm}$ in \cite[Theorem 3.20]{LY2021large}) have the form\[\comp_{n_k}\xra{d_{m_k}^-}\comp_{n_{k-1}}\xla{d_{m_{k-1}}^+}\comp_{n_{k-2}}\xra{d_{m_{k-2}}^-}\cdots\xla{d_{m_k}^+}\comp_{-n_{k}}\]
where $\comp_{i}$ is a summand of $KHI(S^3,K)$ in the Alexander grading $i$. We can omit $d_{r}^\pm$ for $r>1$ and compute $\rk(d_{1}^++d_{1}^-)$. Then $\ish(S^3,K;\C)$ follows from Theorem \ref{thm: formula for singular instanton}. Note that the choice of the knot $K$ or its mirror image $\widebar{K}$ does not matter because we have $\ish(S^3,\widebar{K};\comp)\cong \mathrm{Hom}(\ish(S^3,K;\comp),\comp)$.

\epf

\brem
The definition of instanton L-space knots generalizes to knots in rational homology spheres (e.g.\ the fiber in some Seifert fibered space). There is a more detailed discussion in \cite[\S 5]{LY2021large}. In such cases, one needs to be careful about the choice of the knot or its mirror.
\erem
\brem
Note that torus knots are instanton L-space knots. The instanton homology $\ish(S^3,K)$ and the spectral sequence from $Kh(\widebar{K})$ for a torus knot $K$ were previously studied by many people through various approaches; see Kronheimer--Mrowka \cite{KM2014filtration,KM2022relation}, Hedden--Herald--Kirk \cite{HHK2014traceless}, Poudel--Saveliev \cite{PS2017singular}, and Daemi--Scaduto \cite{daemi2019equivariant}. The computation in Proposition \ref{prop: l space knot computing} would indicate that for non-alternating torus knots, there exist non-vanishing higher differentials in the spectral sequences.
\erem

Another byproduct application of Theorem \ref{thm: formula for singular instanton} is the following non-vanishing result for the next-to-top term.
\begin{proposition}\label{prop: next-to-top nonvanish, main}
	If $K\subset S^3$ is a knot such that
	\[
		\dim I^{\sharp}(S^3, K) \le \dim KHI(S^3, K) + 2\dim KHI(S^3, K, g(K)),
	\]
	then \[KHI(S^3, K, g(K)-1)\neq 0,\]where $g(K)$ and $g(K)-1$ denote the top and next-to-top Alexander gradings.
\end{proposition}

\begin{corollary}
	Suppose $K\subset S^3$ is a knot with unknotting number one, then
	\[
		KHI(S^3, K, g(K)-1)\neq 0.
	\]
\end{corollary}
\bpf
From \cite[Theorem 1.16]{LY2025torsion}, we know that 
\[
	\dim I^{\sharp}(S^3, K) \le \dim KHI(S^3, K) + 3.
\]
From \cite[Proposition 7.16]{kronheimer2010knots} and \cite[Proposition 4.1]{kronheimer2010instanton}, we know that $\dim KHI(S^3, K)\ge 2$ if and only if $K$ is not fibered. Thus, Proposition \ref{prop: next-to-top nonvanish, main} applies to all non-fibered knots with unknotting number one. The fibered case has already been dealt with by \cite[Theorem 1.7]{BS2022khovanov}.
\epf
\brem
With the unoriented (Heegaard) knot Floer homology from Ozsv\'{a}th--Stipsicz--Szab\'{o} \cite{OSS2017unoriented}, we expect our proof about instanton homology can be exported to Heegaard Floer theory as well. The non-vanishing result of the next-to-top term in knot Floer homology was conjectured by Baldwin--Vela-Vick and Sivek \cite[Questions 1.12 and 1.13]{baldwin18fibred} and resolved for the fibered case by Baldwin--Vela-Vick \cite{baldwin18fibred}, the closed $3$-braid case by Chen \cite{Chen2025nexttotop}, and some more general case by Ni \cite{Ni21nexttop}; see also \cite{HW2018geography,Cheng2025positivebraid}.
\erem
\subsection{Relation to Heegaard Floer homology}

The analogue maps of $d_{1}^\pm$ in Heegaard Floer theory were studied by Sarkar \cite[\S 4]{sarkar15moving}, and called $\Psi$ and $\Phi$. Roughly speaking, the maps $\Psi$ and $\Phi$ (considered in the hat version for simplicity) are obtained by counting holomorphic disks passing through the basepoints $w$ and $z$ once, respectively. In Zemke's reconstruction \cite{zemke17moving} (see also \cite[\S 4.2]{Zemke2019link}), these two maps are related to some dividing sets on $K\times I\subset Y\times I$, which are exactly the dividing sets of the contact structures we use to define $d_{1}^\pm$ in Definition \ref{defn: d1 differential}.

The maps $\Psi$ and $\Phi$ can also be regarded as the differentials on the first pages of the spectral sequences from $\widehat{HFK}(S^3,K)$ and $\widehat{HFK}(S^3,-K)$ to $\widehat{HF}(S^3)$, respectively, where $-K$ is the knot $K$ with opposite orientation (not the mirror knot).

Sarkar introduced $\Psi$ and $\Phi$ to study the action on $\widehat{HFK}(S^3,K)$ induced by the basepoint moving around the knot. Up to chain homotopy, the action (over $\ft$) is\[\id+\Psi\circ \Phi\simeq \id+\Phi\circ \Psi.\]The action on the minus version was computed by Zemke \cite{zemke17moving}. An analogous result in instanton theory has not been studied, but the last two authors proved in \cite[\S 6.2]{LY2022integral2} that \[d_{1}^+\circ d_{1}^-=\lambda \cdot d_{1}^-\circ d_{1}^+\]on the homology level for some scalar $\lambda\in\cstar$.

The maps $\Psi$ and $\Phi$ are nonzero for many computed examples. The discussion in \S \ref{sec: Sketch of the proofs} also leads to the following question.

\begin{quest}\label{qu: 2torsion 2}
For any null-homologous nontrivial knot $K$ in $S^3$, are the maps $\Psi,\Phi$ in Heegaard Floer theory, and the differentials $d_{1}^\pm$ in instanton theory always non-vanishing?
\end{quest}
\brem
Since the Borromean knot in the connected sum of two copies of $S^1\times S^2$ has trivial $\Psi,\Phi$ \cite[\S 9]{ozsvath2004holomorphicknot} and trivial $d_{1}^\pm$ \cite[\S 5]{LY2022integral2}, we may replace $S^3$ by a general closed $3$-manifold $Y$ with $b_1(Y)<2$; see Yi Ni's work \cite{Ni21nexttop} for more discussion on this condition.
\erem

Due to the discussion in \S \ref{sec: Sketch of the proofs}, an affirmed answer to the above question implies that the fibered condition in Theorem \ref{thm: knot has 2 torsion} can be removed when $Y=S^3$, and hence we may obtain an affirmative answer to Question \ref{qu: 2torsion} for $S^3$. 

A similar fact in Khovanov theory was mentioned in \cite[Proposition 9.3]{KWZ19khovanov} (see also \cite[Corollary 2.18]{ILM2025Kh}). Roughly speaking, it says that if the first differential in the spectral sequence from the reduced Khovanov homology to the reduced Lee homology is nonzero, then the unreduced Khovanov homology has $2$-torsion. Inspired by \cite{HN2013unlink,LS2022split}, one might try to find some version of a spectral sequence so that Khovanov homology having no $2$-torsion could imply that singular instanton homology has no $2$-torsion, and then an affirmative answer of Question \ref{qu: 2torsion} would resolve Shumakovitch's conjecture; see \cite{Xie2021earring,kronheimer2021barnatan} for some related spectral sequences.

As Yi Ni pointed out to the authors, using the techniques in \cite{Ni21nexttop}, one may prove that if the top Alexander grading summand $\widehat{HFK}(S^3,K,g(K))$ is supported in a single $\intg/2$ homological (Maslov) grading, then both $\Psi$ and $\Phi$ are nonzero. The techniques involve the plus theory and the zero surgery formula in Heegaard Floer homology, which have not been developed in instanton theory (a partial result of the zero surgery formula can be found in \cite[\S 4]{LY2022integral2}). Furthermore, the grading assumption seems to be technical and may be removed in further study. 

Conversely, we can also get some inspiration for Heegaard Floer theory from Theorem \ref{thm: formula for singular instanton}. Rasmussen \cite{Rasmussen2005} conjectured a spectral sequence from the reduced Khovanov homology $\widetilde{Kh}(\widebar{K})$ of the mirror knot $\widebar{K}$ to the knot Floer homology $\widehat{HFK}(S^3,K)$, which was proved by Dowlin \cite{Dowlin18ss} over $\mathbb{Q}$ coefficients and recently by Nahm \cite{nahm2025unoriented,Nahm2025ss} over $\mathbb{Z}/2$ coefficients using the unoriented knot Floer homology $HFK'$ from Ozsv\'{a}th--Stipsicz--Szab\'{o} \cite{OSS2017unoriented}.

Hence $\widehat{HFK}(S^3,K)$ can be regarded as a reduced theory, while there is no obvious definition of the unreduced variant. Baldwin--Levine--Sarkar \cite[\S 3.3]{baldwin17pointed} proposed a candidate for the unreduced knot Floer homology. Roughly speaking, it is $\widehat{HFK}(S^3,K\sqcup U)$ for the split union of $K$ and the unknot $U$, which is isomorphic to two copies of $\widehat{HFK}(S^3,K)$.

Note that $KHI(Y,K)$ is conjectured to be isomorphic to $\widehat{HFK}(Y,K;\comp)$ \cite[Conjecture 7.24]{kronheimer2010knots}. By the exact triangle in Theorem \ref{thm: formula for singular instanton} and the exact sequence about $Kh$ in \cite[Proposition 3.9]{KWZ19khovanov}, some variant of the unoriented knot Floer homology can be regarded as another candidate of unreduced knot Floer homology that is more analogous to $\ish(S^3,K)$ because it does not depend on the orientation of $K$. This idea was also proposed by Nahm \cite{nahm2025unoriented}.
\bdefn
Let $CFK'_2(Y,K)$ be the mapping cone\[\cone(\widehat{CFK}(Y,K)\xra{\Psi+\Phi}\widehat{CFK}(Y,K)),\]which is equivalent to the chain complex $CFK'(Y,K)/U^2$ for the unoriented knot Floer homology $CFK'(Y,K)$ from \cite[Definition 2.1]{OSS2017unoriented}. Define \[HFK'_2(Y,K)=H_*(CFK'_2(Y,K))\]be the homology. 
\edefn
\begin{conj}
    There is an isomorphism \[HFK'_2(Y,K;\mathbb{C})\cong \ish(Y,K;\mathbb{C})\]and a spectral sequence from $Kh(\widebar{K};\mathbb{C})$ to $HFK'_2(Y,K;\mathbb{C})$.
\end{conj}
\brem
The spectral sequence from $Kh(\widebar{K};\mathbb{Z}/2)$ to $HFK'_2(Y,K;\mathbb{Z}/2)$ is already proven by Nahm \cite{Nahm2025ss}, with suitable reinterpretation.
\erem
\brem
The homology $HFK'_2(S^3,K)$ is an analog to $$\widetilde{BN}_2(S^3,K)=H_*(\widetilde{CBN}(S^3,K)/H^2)$$for the reduced Bar-Natan chain complex $\widetilde{CBN}$; see also \cite{Lin2019barnatan,ATZ2023barnatan,kronheimer2021barnatan} for the relation between $\widetilde{BN}_2$ and Floer homology. From \cite[\S 3.5]{KWZ19khovanov}, we only expect $HFK'_2$ is an analogous of $I^\sharp$ over a field where $2$ is invertible. In particular, we do not expect the existence of $2$-torsion in $HFK'_2$. This is why we use the notation $HFK'_2$ instead of $HFK^\sharp$. Inspired by the situation for $Kh$ \cite[Corollary 2.18]{ILM2025Kh} (see also \cite[\S 6.6]{Naot2006Kh}) and $I^\sharp$ \cite[Theorem 8.13]{daemi2019equivariant}, the best candidate of unreduced knot Floer homology $HFK^\sharp$ might be the homology of\[CFK'(Y,K)\otimes_{\Z[U]}\Z^2,\]where $U$ acts on $\Z^2$ by the matrix \[\begin{bmatrix}
0 & 0 \\
2 & 0
\end{bmatrix},\]or equivalently, just the mapping cone\[\cone(\widehat{CFK}(Y,K)\xra{2(\Psi+\Phi)}\widehat{CFK}(Y,K)).\]This candidate satisfies\[\dim HFK^\sharp(Y,K;\mathbb{Z}/2)=2\dim \widehat{HFK}(Y,K;\mathbb{Z}/2)\]and the nontriviality of $\Psi$ or $\Phi$ will imply the existence of $2$-torsion.
\erem
Using the immersed curve techniques developed by Hanselman--Rasmussen--Watson \cite{Hanselman2016,Hanselman2018}, we notice that the construction of $HFK'_2(Y,K)$ is related to the \textit{$N$-filling} in the study of L-space conjecture for graph manifolds \cite{HRRW20graph}. Here $N$ is the twisted $I$-bundle over the Klein bottle, or equivalently, the Seifert fibered manifold $D(2,2)$ over a disk with two orbifold points of order $2$. Let $\lambda$ be the curve on $\partial N$ generating $\ker(H_1(\partial N;\mathbb{Q})\to H_1(N;\mathbb{Q}))$, usually called the \textit{homological longitude}. The $N$-filling of a $3$-manifold $M$ with respect to a curve $\al$ on its toroidal boundary is the closed $3$-manifold obtained by gluing alongthe toroidal boundaries\[M\cup_{\varphi} N~\mathrm{with}~\varphi(\lambda)=\al.\]Note that the choice of $\varphi:\partial N\to \partial M$ is not unique, but $\widehat{HF}(M\cup_\varphi N)$ turns out to be independent of the choice (at least over $\ft$). From \cite[Figure 51]{Hanselman2016} and \cite[Figure 16]{Hanselman2018}, we make the following conjecture.

\begin{conj}
    Let $\mu$ be the meridian of the knot $K\subset Y$. Let $Y_N^\mu$ be the $N$-filling of $Y\backslash {\rm int} N(K)$ with respect to $\mu$. Then we have 
    \[\widehat{HF}(Y_N^\mu)\cong HFK'_2(Y,K)\oplus \widehat{HF}(Y)\oplus \widehat{HF}(Y).\]
\end{conj}
\begin{org}
    The paper is organized as follows. In \S \ref{sec: strategy of the proofs}, we use the octahedral diagram to reduce the proof of Theorem \ref{thm: formula for singular instanton} to some commutative diagram, and we prove Theorem \ref{thm: knot has 2 torsion} and Proposition \ref{prop: next-to-top nonvanish, main} based on Theorem \ref{thm: formula for singular instanton}. In \S \ref{sec: topology setups}, we describe the topology of cobordisms in the commutative diagram. In \S \ref{sec: The commutative diagram}, we prove the commutative diagram by a stretching argument and a model calculation that fixes the sign.

\end{org}
\begin{ack}
The authors thank Charles Stine for his helpful comments on Kirby diagrams. We also thank Peter Kronheimer, Tom Mrowka, Ciprian Manolescu, John A. Baldwin, Jen Hom, Kristen Hendricks, Yi Ni, Masaki Taniguchi, Ali Daemi, Chris Scaduto, Josh Wang, Matt Hogancamp, Onkar Singh Gujral, and Ian Zemke for valuable discussions and comments throughout the long time preparation of this project. The third author is partially supported by Simons Collaboration Grant \#271133 through Peter Kronheimer. The third author thanks Yi Ni for the invitation to Caltech and Yi Liu for the invitation to BICMR at Peking University during this project.
\end{ack}

\section{Strategy of the proofs}\label{sec: strategy of the proofs}

This section is devoted to reducing the proofs of Theorems \ref{thm: knot has 2 torsion} and \ref{thm: formula for singular instanton} and Proposition \ref{prop: next-to-top nonvanish, main} to some commutative diagram, which will be proved in \S \ref{sec: The commutative diagram}. 

\subsection{The octahedral diagram}\label{sec: The octahedral lemma}

The main ingredient of the exact triangle in Theorem \ref{thm: formula for singular instanton} is the octahedral lemma for the derived category of chain complexes over $\C$ (for example, see \cite[Proposition 10.2.4]{homologicalalgebra94} and \cite[Lemma A.3.10]{ozsvathbookgrid}).

For simplicity, we adopt the setups from \cite[\S 2]{LY2025torsion} for the framed instanton homology $\ish(Y)$, the sutured instanton homology $SHI(M,\ga)$, and the (sutured) instanton knot homology $KHI(Y,K)$. 

In particular, for a framed knot $K\subset Y$, we write $Y\bbslash K=Y\backslash{\rm int}N(K)$. Let $\Gamma_n$ be the suture on $\partial (Y\bbslash K)\cong T^2$ consisting of a pair of oppositely oriented non-separating simple closed curves of slope $-n$ (the minus sign is chosen to be consistent with the notations in the last two authors' previous work). Let $\Gamma_{\mu}=\mu\cup(-\mu)$ consist of two meridians with opposite orientations. As defined in \cite[\S 7]{kronheimer2011knot}, we have\begin{equation}\label{eq: I sharp U}
    \ish(Y;\C)\cong KHI(Y,U)
\end{equation}for the unknot $U$, and
\begin{equation}\label{eq: KHI}
    KHI(Y,K) = SHI(Y\bbslash K,\Gamma_{\mu})\aand KHI(Y_n(K),\widetilde{K}_n) = SHI(Y\bbslash K,\Gamma_{-n}),
\end{equation}
for the dual knot $\widetilde{K}_n$ in the manifold $Y_n(K)$ obtained from $Y$ by $n$-surgery along $K$.

Note that we can use the Floer's excision cobordism in \cite[\S 5.4]{kronheimer2011khovanov} to identify $KHI(Y,K)$ with $I^\natural(Y,K;\C)$, which is functorial with respect to cobordism maps that are supported in the region away from the neighborhood of $K$; see \cite[\S 2]{Ye2025dualknot} for more details.

Then we form the following octahedral diagram and will show it satisfies the assumption of the octahedral lemma.

\begin{equation}\label{eq: octahedral used 1}
\xymatrix{
	&\ish (Y,K;\comp)\ar@{..>}[dr]\ar[ddl]&\\
	KHI(Y,K) \ar@{..>}[ur]\ar[d]&&KHI(Y,K)\ar@{..>}[ll]^{h\circ l}\ar[ddl]^>>>>>>>>>>>>{l}\\
\ish (Y_0(K);\comp)\ar[dr]_{f}\ar[rr]^{g\circ f}&&\ish (Y_2(K);\comp)\ar[u]\ar[uul]\\
	&KHI(Y_1(K),\widetilde{K}_1)\ar[uul]^<<<<<<<<<<<<{h}\ar[ur]_{g}&
	}
    \end{equation}
    


Note that the fourth exact triangle is exactly what we want in Theorem \ref{thm: formula for singular instanton}. We will also show the composition map $h\circ l$ is the map $c_+d_{1}^++c_-d_{1}^-$. Since we can pick any framing of $K$, the surgery coefficients $0,1,2$ in (\ref{eq: octahedral used 1}) can be replaced by $n,n+1,n+2$ for any $n\in\intg$.

\subsection{Triangles in sutured instanton homology}\label{sec: Triangles in sutured instanton homology}

The triangles in (\ref{eq: octahedral used 1}) associated with $f$ and $g$ follow from the work of the last two authors \cite[Lemma 4.9]{LY2020} by choosing suitable bases for the torus boundary. To be self-contained, we state the proof for the reader's convenience.

\bprop\label{prop: -1 triangle}
Suppose $K$ is a framed knot in a closed 3-manifold $Y$. Then there exist exact triangles
\begin{equation}\label{eq: H dual}
\xymatrix{
SHI(\Ga_{-1})\ar[rr]^{H_\al}&&SHI(\Ga_{\mu})\ar[dl]^{G_{\al}}\\
&\ish(Y_0(K);\comp)\ar[ul]^{F_{\al}}&
}	
\xymatrix{
SHI(\Ga_{\mu})\ar[rr]^{H_\be}&&SHI(\Ga_{-1})\ar[dl]^{G_{\be}}\\
&\ish(Y_2(K);\comp)\ar[ul]^{F_{\be}}&
}	
\end{equation}where we omit $Y\bbslash K$ in $SHI$ for simplicity.
\eprop
\brem\label{rem: minus sign}
Note that in the statement of \cite[Lemma 4.9]{LY2020}, we add minus signs for manifolds and sutures to denote the opposite orientations. Those notations were used to make the statement compatible with contact gluing maps (especially bypass maps), which are not essential in the proof of \cite[Lemma 4.9]{LY2020}, and indeed in \cite[Lemma 3.21]{LY2020}. If we apply the lemma to the manifolds with opposite orientations, we will obtain a statement about the manifolds without minus signs (two minus signs just cancel out). Alternatively, we may also take the dual spaces in the original statement with minus signs because $SHI(-M,-\ga)$ is canonically isomorphic to $\operatorname{Hom}(SHI(M,-\ga),\mathbb{C})$ by \cite[Theorem 1.2 (3)]{li2018gluing}, and $SHI(M,-\ga)$ is canonically isomorphic to $SHI(M,\ga)$ by the proof of \cite[Lemma 2.5]{BS2022khovanov}. Unpacking the definitions of maps and isomorphisms, one can show that these two constructions are equivalent.
\erem
\bpf[Proof of Proposition \ref{prop: -1 triangle}]
Let $\al^\p$ and $\be^\p$ be the curves inside $\ybnk$ obtained by pushing the curves \[\al=\lambda,\be=2\mu+\lambda\subset \pybnk\] into the interior, respectively. Suppose $\al^\p$ and $\be^\p$ are framed by the surface framing from $\pybnk$. Then the two exact triangles follow from the surgery exact triangles along $\al^\p$ and $\be^\p$ with framings $(\infty,0,1)$, respectively: the $\infty$-surgery does not change anything; the $1$-surgery is equivalent to a Dehn twist on $\pybnk$ along $\al$ or $\be$, which twists the suture; the $0$-surgery is equivalent to a (contact) $2$-handle attachment along $\al$ or $\be$ and hence fills the knot which leads to the sutured manifold $(Y_0(K),\delta)$ or $(Y_2(K),\delta)$ (\emph{cf.} \cite[\S 3.3]{baldwin2016instanton}). Indeed, the $0$-surgery makes the knot complement become the complement of an unknot inside $Y_0(K)$ or $Y_2(K)$, and the suture becomes the meridians of the unknot. From (\ref{eq: I sharp U}), we know that $0$-surgery gives $\ish(Y_0(K);\comp)$ or $\ish(Y_2(K);\comp)$. 

Note that the original surgery exact triangle in instanton theory has some extra bundle sets in the surgery manifolds and cobordisms (cf.\ \cite[Theorem 1.2]{scaduto2015instanton}), but one can cancel some of them using the strategy in \cite[\S 4]{BS2022khovanov} and \cite[\S 8]{LY2025dimension}. More precisely, we need to use a closure of a balanced sutured manifold in which $\al^\p$ (or $\be^\p$ for the second triangle) bounds a punctured torus in the closed $3$-manifold and the outgoing framing of the punctured torus is exactly the surgery framing of $\al^\p$. Then \cite[Corollary 8.2]{LY2025dimension} implies that the bundle sets in all three closed $3$-manifolds of the triangle and the cobordisms associated to $G_\al$ and $F_\al$ can be canceled. Note that the bundle set in the cobordism associated to $H_\al$ is nontrivial, which is the union of the cocore disk and the punctured torus, and has self-intersection $1$. We cannot further use the naturality trick as in \cite[\S 4]{BS2022khovanov} to cancel the bundle set in $H_\al$ because now the three vertices in the triangle are not symmetric and it might be hard to construct the punctured torus with the framing property in the closures of the other two balanced sutured manifolds.
\epf

\brem\label{rem: new curve}
If we start with the dual curve $\al^\pp$ of $\al^\p$ inside $(Y\bbslash K, \Ga_1)$ with the induced framing in the proof of Proposition \ref{prop: -1 triangle}, then the exact triangle should be associated to framings $(\infty,-1,0)$, i.e.\, the map $H_\al$ is obtained by the $(-1)$-surgery along $\al^\pp$. We also define $\be^\pp$ to be the dual curve of $\be^\p$ inside $(Y\bbslash K,\Ga_\mu)$ with the induced framing. Then the map $H_\be$ is obtained by the $(-1)$-surgery along $\be^\pp$. 
\erem

Then we can use the bypass maps introduced in \cite[\S 4]{BS2022khovanov} to describe $H_\al$ and $H_\be$ as follows.

Recall that bypass maps are obtained by composing the maps associated with a contact $1$-handle attachment and then a contact $2$-handle attachment in a specific way related to a bypass arc intersecting the suture at three points. In particular, they are contact gluing maps, which are maps between sutured manifolds with opposite orientations on the manifolds and the sutures. As mentioned in Remark \ref{rem: minus sign}, we can start with the sutured manifolds with opposite orientations and apply the results about bypass maps. Or equivalently, we can take the original bypass maps and consider the induced maps between the dual spaces.

Following the notations in \cite[\S 4.2]{LY2020}, let \[\psi_{+,\mu}^{-1},\psi_{-,\mu}^{-1}:SHI(-\ybnk,-\Ga_\mu)\to SHI(-\ybnk,-\Ga_{-1})\]
\[\psi_{+,-1}^{\mu},\psi_{-,-1}^{\mu}:SHI(-\ybnk,-\Ga_{-1})\to SHI(-\ybnk,-\Ga_{\mu})\]be positive and negative bypass maps, where $-1$ does not denote the inverse map but the suture $\Ga_{-1}$. Due to the above discussion, we define \[(\psi_{+,\mu}^{-1})^\vee,(\psi_{-,\mu}^{-1})^\vee:SHI(\ybnk,\Ga_{-1})\to SHI(\ybnk,\Ga_{\mu})\]
\[(\psi_{+,-1}^{\mu})^\vee,(\psi_{-,-1}^{\mu})^\vee:SHI(\ybnk,\Ga_{\mu})\to SHI(\ybnk,\Ga_{-1})\]between the dual spaces. By the previous work of the last two authors and Remark \ref{rem: new curve}, we have the following lemmas.

\blem[{\cite[Proposition 4.1]{LY2022integral1}}]\label{lem: -1 surgery}
Let $H_\al,H_\be$ be maps in (\ref{eq: H dual}). Then there exist scalars $c_1,c_2,c_3,c_4\in\cstar$ so that \[H_\al=c_1(\psi_{+,-1}^{\mu})^\vee+c_2(\psi_{-,-1}^{\mu})^\vee\aand H_\be=c_3(\psi_{+,\mu}^{-1})^\vee+c_4(\psi_{-,\mu}^{-1})^\vee.\]
\elem
\brem
In the proof of \cite[Proposition 4.1]{LY2022integral1}, the last authors had not considered the bundle set in $H_\al$ carefully and only studied the case where the cobordism associated to $H_\al$ has no extra bundle set. However, the extra bundle set from the proof of Proposition \ref{prop: -1 triangle} does not affect the result. Roughly, the proof of \cite[Proposition 4.1]{LY2022integral1} relies on cutting out the neighborhood $V\cong S^1\times D^2$ of $\al^\p$ and gluing back via contact gluing map. The extra bundle set lies in the closure of $(V,\Ga_1)$ and hence all proofs apply verbatim.
\erem
\blem[{\cite[Corollary 4.37]{LY2020}}]\label{lem: composition of bypass maps}
We have\[\psi_{+,\mu}^{-1}\circ \psi_{-,-1}^{\mu}=0\aand \psi_{-,\mu}^{-1}\circ \psi_{+,-1}^{\mu}=0.\]
\elem
\brem
In \cite[\S 4]{LY2020}, we assumed $K$ is (rationally) null-homologous for simplicity. This condition is not necessary for our purposes because \cite[Corollary 4.37]{LY2020} follows from the bypass exact triangle \cite[\S 4]{BS2022khovanov} and \cite[Lemma 4.34]{LY2020}, and the latter lemma follows from Honda’s classification of tight contact structures on $T^2\times I$ \cite{honda2000classification}. Those results do not rely on the assumption that $K$ is (rationally) null-homologous.
\erem
\bdefn[{\cite[\S 3.4]{LY2021large} and \cite[Formula (6.2)]{LY2022integral2}}]\label{defn: d1 differential}
Define \[d_{1}^\pm\deq (\psi_{\pm,-1}^{\mu})^\vee\circ (\psi_{\pm,\mu}^{-1})^\vee: SHI(\ybnk,\Ga_\mu)\to SHI(\ybnk,\Ga_\mu).\]
\edefn

The following corollary follows from Lemma \ref{lem: -1 surgery} and Proposition \ref{lem: composition of bypass maps} directly.
\bcor\label{cor: composition of H maps}
Let $c_+=c_3c_1$ and $c_-=c_2c_4$. We have
\[H_\al\circ H_\be=c_3c_1(\psi_{+,-1}^{\mu})^\vee\circ (\psi_{+,\mu}^{-1})^\vee+c_4c_2 (\psi_{-,-1}^{\mu})^{\vee}\circ (\psi_{-,\mu}^{-1})^\vee=c_+ d_{1}^++c_- d_{1}^-.\]
\ecor

Note that $H_\al=h$ and $H_\be=l$ in the octahedral diagram (\ref{eq: octahedral used 1}). We will use this fact in the proof of Theorem \ref{thm: formula for singular instanton}.

\subsection{Triangles in singular instanton homology}\label{sec: Triangles in singular instanton homology}

The triangle associated with $g\circ f$ in (\ref{eq: octahedral used 1}) follows from the work of the first author \cite{bhat2023newtriangle}, which indeed also works over an arbitrary coefficient ring.
\bprop[{\cite[Theorem 5.1 and Remark 7.5]{bhat2023newtriangle}}]\label{prop: original triangle}
Suppose $K$ is a framed knot in a closed 3-manifold $Y$. Suppose $f_0 = f_0^+ - f_0^-$ is defined in \cite[\S 7.2]{bhat2023newtriangle}. Then for any commutative ring $R$, there exists an exact triangle
\begin{equation}\label{eq: singular instanton triangle 0}
\xymatrix{
\ish(Y_0(K);R)\ar[rr]^{f_0}&&\ish(Y_2(K);R)\ar[dl]^{}\\
&\ish(Y,K;R)\ar[ul]^{}&
}	
\end{equation}
\eprop
Note that $f_0^\pm$ are defined by cobordism with a cone point over $L(2,1)\cong \mathbb{RP}^3$. We can interpret them as the usual instanton cobordism maps as follows.

\bprop\label{prop: interpretation of f_0}
Suppose $K$ is a framed knot in a closed 3-manifold $Y$. For $n\in\mathbb{Z}$, let \[W^n_{n+1}:Y_n(K)\to Y_{n+1}(K)\]be the surgery cobordism and let $W^n_{n+2}=W^{n+1}_{n+2}\circ W^n_{n+1}$ be the composition cobordism, as shown in Figure \ref{fig: W}. Let $\Omega\subset W^n_{n+2}$ be the union of the cocore disk in $W^n_{n+1}$ and the core disk in $W^{n+1}_{n+2}$, which is a $2$-sphere of self-intersection $-2$ formed by the Seifert disk bounded by $\alpha$ in Figure \ref{fig: W}, capped off by the core of the $2$-handle. Then we have\[f_0^+=I^\sharp(W^0_{2})\aand f_0^-=I^\sharp_\Omega (W^0_{2}),\]where the subscript $\Omega$ denotes the bundle data.
\eprop
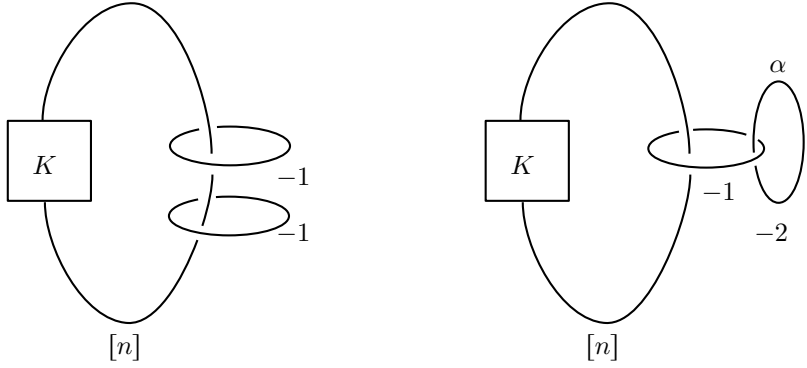
\begin{figure}
\begin{minipage}{0.4\textwidth}
	\begin{overpic}[width = 0.8\textwidth]{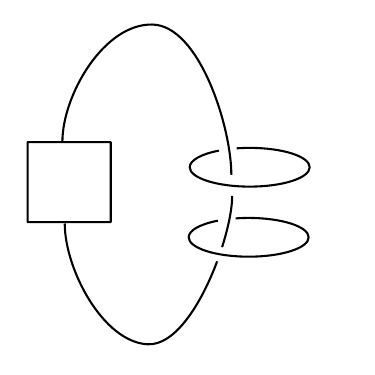}
		\put (14,48){$K$}
		\put (34,0){$[n]$}
		\put (79,45){$-1$}
		\put (79,30){$-1$}
	\end{overpic}
\end{minipage}
\begin{minipage}{0.4\textwidth}
	\begin{overpic}[width = 0.8\textwidth]{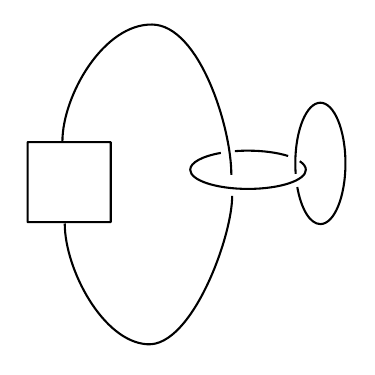}
		\put (14,48){$K$}
		\put (34,0){$[n]$}
		\put (65,41){$-1$}
		\put (79,30){$-2$}
		\put (83,75){$\alpha$}
	\end{overpic}
\end{minipage}
\caption{Two Kirby diagrams of $W^n_{n+2}$ related by an obvious handle-slide. The rightmost unknot in the right subfigure is denoted by $\alpha$.}\label{fig: W}
\end{figure}

\bpf
For the purpose of this proof, let $W = W^0_2$, $X$ be the tubular neighborhood of $\Omega$, and $M = W \backslash X$.

We also set some conventions only relevant to this proof. $\mcm^\sharp_{\omega} (Z)$ will be used to refer to the moduli space of ASD connections (modulo determinant 1 gauge transformations) on the pair $(Z, H \times \gamma)$ with bundle data described by $\omega \cup (b \times \gamma)$. Here, $H$ denotes the hopf-link, $b$ is an arc connecting the two components of $H$, and $\gamma$ is a path implicitly specified when defining $I^\sharp(Z)$. 

Neck stretching along $\partial X \cong \rp$, we see that the counts of the $0$-dimensional moduli spaces are related as
\begin{equation} \label{eq: reinterpretation}
\# \mcm^\sharp_\omega(W) = \# (\mcm^\sharp(M) \times_{\chi(\rp)} \mcm_\omega (X)),
\end{equation}
for all $\omega \in \{ \varnothing, \Omega \}$. We make a few remarks at this stage. There is a potential gluing-parameter in play above which is not reflected in the notation. Next, while we can use perturbations to ensure the moduli $\mcm^\sharp (M)$ is transverse, a similar fix may not apply to $\mcm_\omega (X)$ if it contains reducibles; however, since $H^+(X; \R) = 0$, this is a non-issue if the elements are central connections. Lastly, $\chi (\rp)$ is the $SU(2)$ character variety; it consists of two points $\theta_\pm$ where we denote by $\theta_+$ the trivial flat connection on $\rp$.

To finish the proof, we must show that $\mcm_\omega (X)$ consists of a single point (counted with positive sign) for each $\omega \in \{ \varnothing, \Omega \}$ and that the restriction of this solution to $\partial X \cong \rp$ is equal to $\theta_+$ and $\theta_-$ when $\omega = \varnothing$ and $\omega = \Omega$, respectively. 

Let $A_M$ be a connection on $M$ and $A_X$ be a connection on $X$ with restrictions to $\rp$ that agree. If these give rise to an element of the $0$-dimensional moduli space on the right of \eqref{eq: reinterpretation}, we must have that $A_X$ is central; for otherwise there would be a gluing parameter, since every element of $\chi (\rp)$ is central, which would give a positive dimensional family inside the fiber product. Thus, $A_X$ is forced to be a central flat connection. If $\omega = \varnothing$, then since $X$ is simply connected, $A_X$ must be trivial. If $\omega = \Omega$, we need to find $SU(2)$ representations of $\pi_1 (X \backslash \Omega)$ such that $m_\Omega$, the meridian of $\Omega$ in $X$, is mapped to $-\id$. Since $m_\Omega$ generates this group, $A_X$ is uniquely specified. Hence, in both situations, we have a unique element of the moduli $\mcm_\omega(X)$ that contributes to \eqref{eq: reinterpretation} and given the descriptions above, their restrictions to $\rp$ are as desired.

Finally, since $A_X$ is central in the context above, both are counted with positive signs in \eqref{eq: reinterpretation}, which completes the proof.
\epf

\subsection{Proofs of the main results}\label{sec: Proofs of the main theorems}

The remaining sections are aimed to prove the following proposition.

\bprop\label{prop: commutative}
Suppose the maps $F_\al,G_\be$ and $f_0$ are from Propositions \ref{prop: -1 triangle} and \ref{prop: original triangle}, respectively. Then we have the following commutative diagram up to sign:
\begin{equation*}
	\xymatrix{\ish(Y_0(K);\comp)\ar[drr]_{F_\al}\ar[rrrr]^{f_0}&&&&\ish(Y_2(K);\comp)\\
&&SHI(\Ga_{-1})\ar[urr]_{G_{\be}}&&
	}
\end{equation*}
\eprop

Assuming Proposition \ref{prop: commutative}, we finish the proofs of the main results in the introduction.

\bpf[Proof of Theorem \ref{thm: formula for singular instanton}]
We consider the octahedral diagram (\ref{eq: octahedral used 1}). The exact triangles in Propositions \ref{prop: -1 triangle} and \ref{prop: original triangle}, together with the commutative diagram in Proposition \ref{prop: commutative}, show that the diagram satisfies the assumption of the octahedral lemma. Then we obtain the fourth exact triangle, where the map $h\circ l$ is described by Corollary \ref{cor: composition of H maps}. This concludes the proof of the exact triangle (\ref{eq: exact trianlge formula}).

If $K$ is rationally null-homologous and there is a (rational) Seifert surface of $K$, we can construct a $\mathbb{Z}$-grading on $KHI(Y,K)$ by the method in \cite{li2019direct,li2019decomposition}. Following \cite[(3.8)]{LY2021large}, we know $d_{1}^\pm$ are homogeneous with different grading shifts. By an easy algebraic lemma (cf.\ \cite[Lemma 2.23]{LY2022integral1}), we know the mapping cones for different scalars are isomorphic.
\epf

\bpf[Proof of Theorem \ref{thm: knot has 2 torsion}]
From Theorem \ref{thm: formula for singular instanton} and the discussion in \S \ref{sec: Sketch of the proofs}, we have \begin{equation}\label{eq: iff 2}
    \begin{aligned}
        \dim \ish(Y,K;\F_2)=&2\dim \ina(Y,K;\F_2)\\\ge& 2\dim \ina(Y,K;\comp)\\=&2\dim KHI(Y,K)\\\ge& \dim \ish(Y,K;\comp)+\rk(c_+d_{1}^++c_-d_{1}^-).
    \end{aligned}
\end{equation}
Hence $\ish(Y,K;\intg)$ has $2$-torsion if
$\rk(c_+d_{1}^++c_-d_{1}^-)>0$.

If $K$ is a null-homologous knot, then we can compare notations in \cite[\S 1.4]{BS2022khovanov} and \cite[\S 4.2]{LY2020}. In particular, we obtain
\[\phi_0^{SV}=\psi_{+,\mu}^0:SHI(-\ybnk,-\Ga_{0})\to SHI(-\ybnk,-\Ga_{\mu})\aand\]
\[C=\psi_{+,0}^{\mu}:SHI(-\ybnk,-\Ga_{\mu})\to SHI(-\ybnk,-\Ga_{0}).\]
By \cite[Lemma 4.34]{LY2020}, we have \[\begin{aligned}
    \phi_0^{SV}\circ C=&\psi_{+,\mu}^0\circ \psi_{+,0}^\mu\\=&\psi_{+,\mu}^{0}\circ \psi_{-,0}^{-1}\circ \psi_{+,-1}^\mu\\=&\psi_{+,\mu}^{-1}\circ \psi_{+,-1}^\mu\\=&d_1^+
\end{aligned}\]
Moreover, we suppose $K\subset Y$ is a fibered knot of genus $g>0$. If $Y\not\cong \#^{2g}S^1\times S^2$ and the monodromy from the fibration of $K$ is not right-veering, then the discussion before \cite[Theorem 1.22]{BS2022khovanov} implies that there exists a nonzero element $x$ in the Alexander grading $g-1$ of $SHI(-\ybnk,-\Ga_{\mu})$ so that $\phi_0^{SV}\circ C(x)\neq 0$ in the Alexander grading $g$. Since $d_{1}^+$ and $d_{1}^-$ have different grading shifts, from Definition \ref{defn: d1 differential}, we have \[\rk(c_+d_{1}^++c_-d_{1}^-)\ge \rk (d_1^+)=\rk(\phi_0^{SV}\circ C)\ge 1.\]Hence $\ish(Y,K;\intg)$ has $2$-torsion.

If $K$ is right-veering, then the mirror knot inside $-Y$ is not right-veering. We can carry out the proof for the mirror knot and then apply the duality of singular instanton homology to show $\ish(Y,K;\intg)$ has $2$-torsion.
\epf

\begin{proof}[Proof of Proposition \ref{prop: next-to-top nonvanish, main}]
	Let $\delta = d_1^+ + d_1^-$ and $\delta_\lambda=d_1^+-\lambda\cdot d_1^-$ for $\lambda\in\cstar$. We have two basic properties:
	\begin{itemize}
		\item For some $\lambda$ that might depend on the knot, we have \[\delta\circ \delta_\lambda = (d_1^+)^2-\lambda\cdot (d_1^-)^2+(d_1^-\circ d_1^+-\lambda\cdot  d_1^+\circ d_1^-)=0,\]which follows from \cite[Theorem 3.20]{LY2021large} and \cite[Theorem 6.5]{LY2022integral2}.
		\item For the Alexander grading $i$, we have\[\delta(KHI(S^3,K,i)),~\delta_\lambda(KHI(S^3,K,i))\subset KHI(S^3,K,i+1)\oplus KHI(S^3,K,i-1),\]which follows from the grading shifts of $d_1^\pm$ (cf.\ \cite[(3.8)]{LY2021large}).
	\end{itemize}
	We prove the proposition by contradiction. If $KHI(S^3,K,g(K) - 1) = 0$, then we have
	\[
		KHI(S^3, K) = KHI(S^3,K,g(K)) \oplus A \oplus  KHI(S^3,K,-g(K))
	\]
	where we define
	\[
		A =\bigoplus_{i = 2- g(K)}^{g(K) - 2} KHI(S^3,K,i).
	\]
    Note that $\dim KHI(S^3, K)$ is odd and $\dim KHI(S^3,K,g(K))=\dim KHI(S^3,K,-g(K))$ (cf. \cite[Theorem 3.20]{LY2021large}, \cite[Corollary 1.4]{scaduto2015instanton}, and \cite[Remark 3.11]{BS2022khovanov}). Then $\dim A$ is also odd.
    
	The fact that $\delta$ shifts grading at most by $1$ and the vanishing assumption of $KHI(S^3,K,g(K)-1)$ imply that $\delta(A),\delta_\lambda(A)\subset A$ and $\delta(KHI(S^3,K,\pm g(K))) = 0$. From $\delta\circ\delta_\lambda=0$, we know that
    \[\im (\delta_\lambda|_A)\subset \ker (\delta|_A).\]Together with \[A/\ker (\delta|_A)\cong \im (\delta|_A),\]we have
    \[\dim A-\dim \ker (\delta|_A)=\dim \im(\delta|_A)=\dim \im (\delta_\lambda|_A)\le \dim  \ker (\delta|_A),\]which implies
	\[
		\dim {\rm ker}(\delta|_A)\geq \frac{\dim A + 1}{2}
	\]by the odd dimension fact of $A$.
    
	Thus, by Theorem \ref{thm: formula for singular instanton} and the assumption $KHI(S^3,K,g(K) - 1) = 0$, we conclude that
	\[
	\begin{aligned}
		\dim I^{\sharp}(S^3, K) &= \dim H_*(\cone(\delta))\\&=4\dim  KHI(S^3, K, g(K)) + \dim H_*(\cone(\delta|_A))\\&=4\dim  KHI(S^3, K, g(K))+2\cdot \dim\ker(\delta|_A)\\
		&\geq 4\dim KHI(S^3,K,g(K))+\dim A + 1\\
		&= \dim KHI(S^3,K) + 2\dim  KHI(S^3, K, g(K)) +1.
	\end{aligned}
	\]
	Hence we derive a contradiction.
\end{proof}

\section{Topology of cobordisms}\label{sec: topology setups}
In this section, we describe the topology of cobordisms, which will be used in \S \ref{sec: The commutative diagram} for the proof of Proposition \ref{prop: commutative}.
\subsection{Cobordisms}
 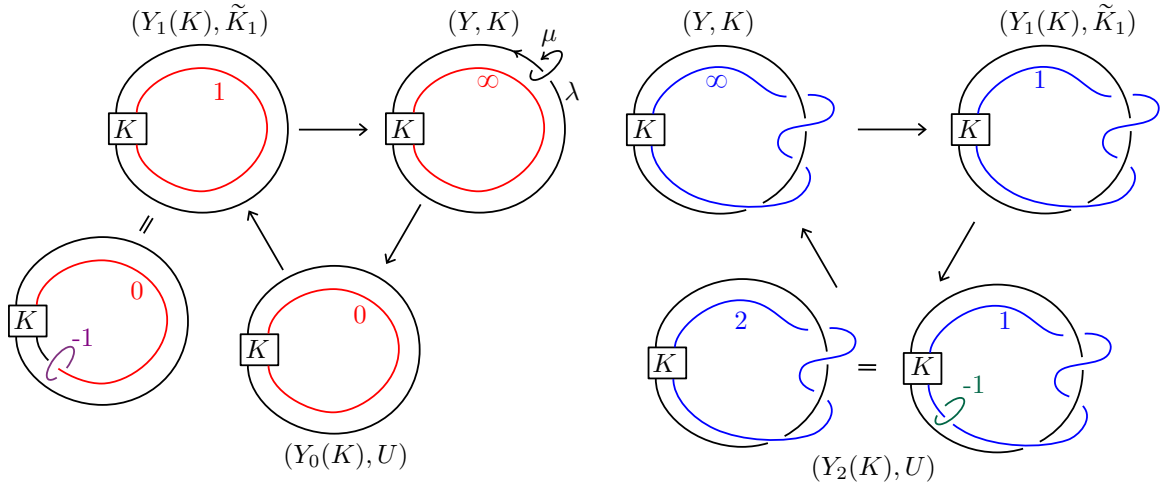
\begin{figure}[htbp]
 \centering
  \begin{overpic}[width=\textwidth]{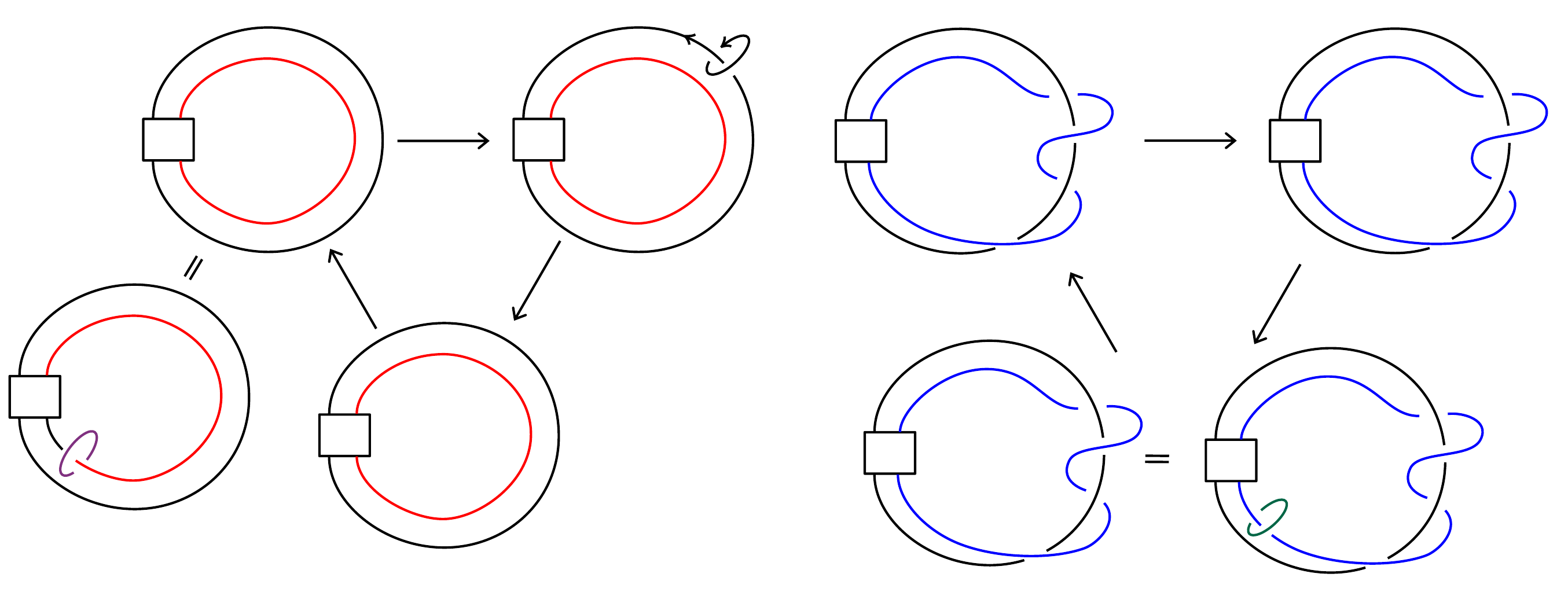}
  	\put(1,12.5){$K$}
  	\put(9.5,28.8){$K$}
  	\put(11,38){$(Y_1(K),\widetilde{K}_1)$}
  	\put(18,32){$\color{red}1$}
  	\put(11,15){$\color{red}0$}
  	\put(6,11){\color{cpurple}-$1$}
 	
  	\put(20.8,10){$K$}
  	\put(24,1){$(Y_0(K),U)$}
  	\put(30,13){\color{red}$0$}
 	
  	\put(33.2,28.8){$K$}
  	\put(38,38){$(Y,K)$}
  	\put(46,37){$\mu$}
  	\put(48,32){$\lambda$}
  	\put(40.5,33){\color{red}$\infty$}

  	\put(53.8,28.8){$K$}
  	\put(58,38){$(Y,K)$}
  	\put(60,33){\color{blue}$\infty$}
 	
  	\put(81.5,28.8){$K$}
  	\put(85,38){$(Y_1(K),\widetilde{K}_1)$}
  	\put(88,33){\color{blue}$1$}
 	
  	\put(55.5,8.8){$K$}
  	\put(77.5,8.5){$K$}
  	\put(62.5,12.5){\color{blue}$2$}
  	\put(85,12.5){\color{blue}$1$}
  	\put(82,7){\color{lygreen}-$1$}
  	\put(69,0){$(Y_2(K),U)$}
  \end{overpic}

	\caption{The Kirby diagrams associated with the exact triangles}\label{fig: surgery diagram 1}
\end{figure}

We first describe the cobordism associated with the composition $G_\be\circ F_\al$ as in the proof of Proposition \ref{prop: -1 triangle}. From (\ref{eq: I sharp U}), we replace $\ish(Y;\comp)$ by $KHI(Y,U)$. Also, we recall from (\ref{eq: KHI}) that \[SHI(\Ga_\mu)= KHI(Y,K)\aand SHI(\Ga_{-1})=KHI(Y_1(K),\widetilde{K}_1).\]

From the proof of Proposition \ref{prop: -1 triangle}, the maps $F_\al$ and $G_\be$ are obtained by surgery cobordisms, where the surgery curves are away from the neighborhood of the knot. Hence, we use a curve with a box to denote the knot $K$ and draw the surgery curves in different colors. We draw the Kirby diagrams for the triangles  (\ref{eq: H dual}) as in Figure \ref{fig: surgery diagram 1}. Note that in the first triangle, we consider the $(\infty,0,1)$ surgery along $\al^\p$, but in the second triangle, we consider the $(\infty,-1,0)$ surgery along $\be^\pp$ as in Remark \ref{rem: new curve}. We use the purple curve and the green curve to denote the cobordisms associated with $F_\al$ and $G_\be$, respectively. Here we consider the blackboard framings of knots. Note that if the surface framing of $\be^\pp$ is $-1$, then the blackboard framing is $-1+2=1$. 

 \begin{figure}[htbp]
 \centering
 \begin{overpic}[width=0.8\textwidth]{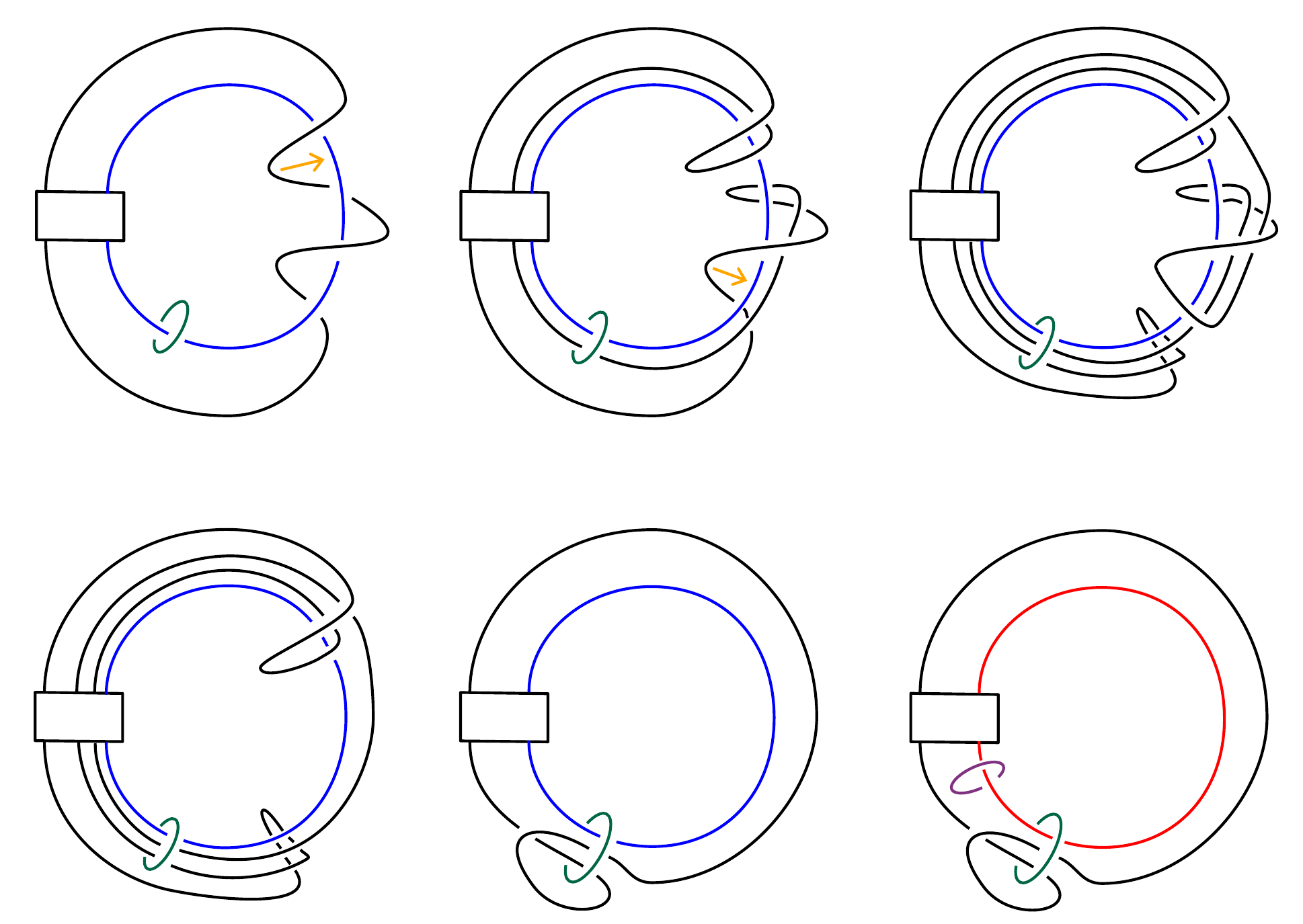}
 	\put(5,53.5) {$K$}
 	\put(37.5,53.5) {$K$}
 	\put(72,53.5) {$K$}
 	\put(5,15) {$K$}
 	\put(37.5,15) {$K$}
 	\put(72,15) {$K$}
 	
 	\put(27,63) {$n$}
 	\put(60,63) {$n-3$}
 	\put(95,63) {$n-4$}
 	\put(95,25) {$n-4$}
 	
 	\put(11,60) {\color{blue} $1$} 
 	\put(13,48.5) {\color{lygreen} $-1$}
 	\put(43.5,60) {\color{blue} $1$} 
 	\put(45.5,48) {\color{lygreen} $-1$}
 	\put(78,60) {\color{blue} $1$} 
 	\put(79,47.5) {\color{lygreen} $-1$} 
 	\put(11,22) {\color{blue} $1$} 
 	\put(12,9) {\color{lygreen} $-1$}
 	\put(44,22) {\color{blue} $1$} 
 	\put(45,9) {\color{lygreen} $-1$}
 	\put(79,22) {\color{red} $0$} 
 	\put(81,9) {\color{lygreen} $-1$}
 	\put(77,12) {\color{cpurple} $-1$}
 \end{overpic}
	\caption{Handle-sliding}\label{fig: surgery diagram 2}
\end{figure}
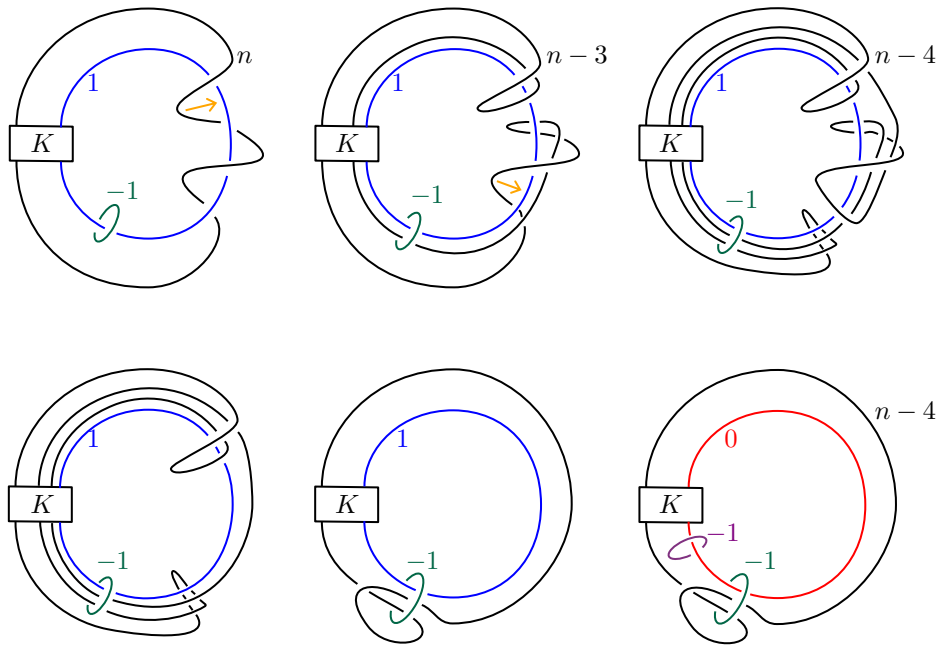

Note that there are two Kirby diagrams of $(Y_1(K),\widetilde{K}_1)$. We can identify the right one with the left one by handle-sliding the knot over the blue curve twice as in Figure \ref{fig: surgery diagram 2}. Then we can replace the blue curve with the red and purple curves to draw the cobordisms associated with $F_\al$ and $G_\be$ in the same diagram. If the original knot has blackboard framing $n$, then the knot after handle-sliding has blackboard framing $n-4$, though we do not need this framing fact in our proof.

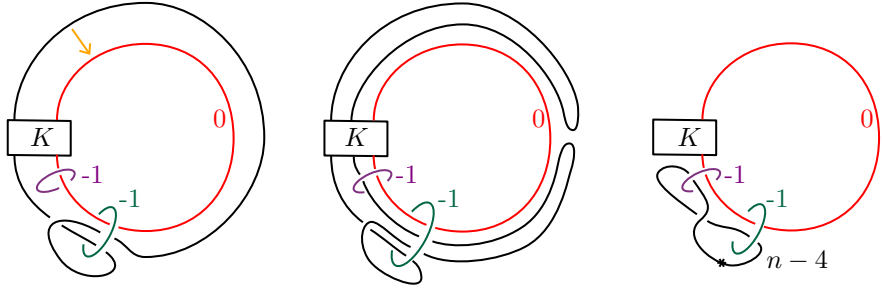
\begin{figure}[htbp]
\centering
	\begin{overpic}[width=0.8\textwidth]{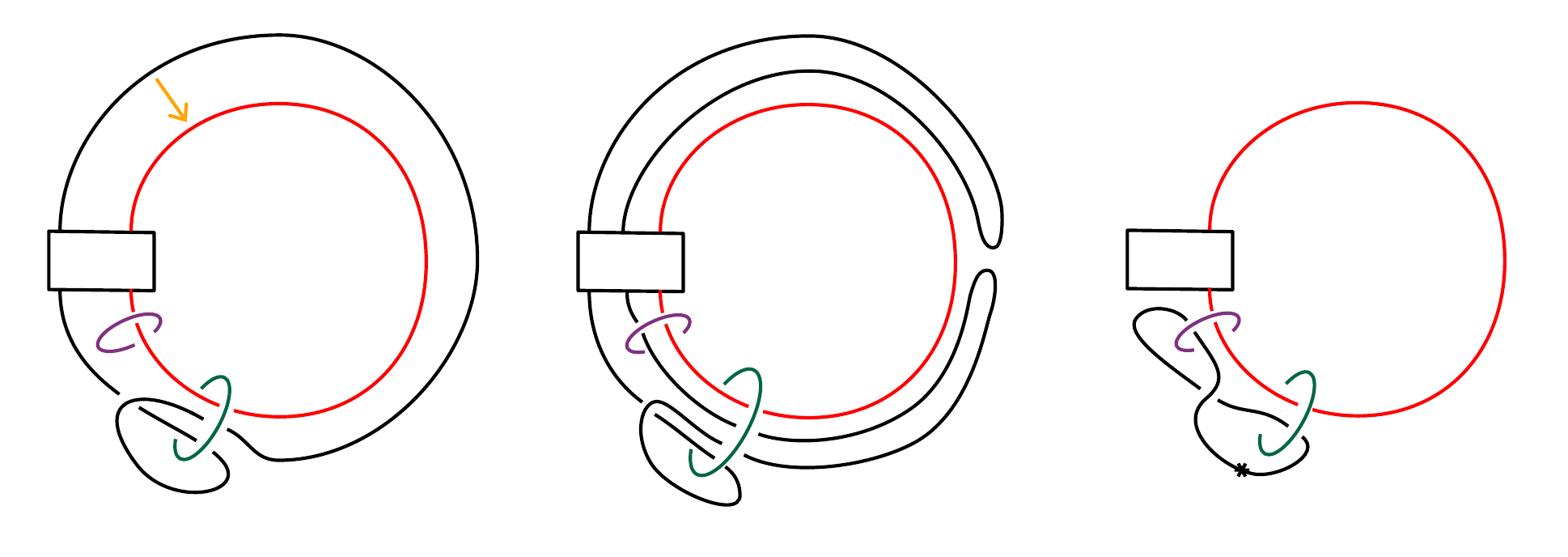}
		\put (11,14) {\color{cpurple} -1}
		\put (15, 11.5) {\color{lygreen} -1}
		\put (25,20) {\color{red} 0}
		\put (44.5,14) {\color{cpurple} -1}
		\put (49, 11.5) {\color{lygreen} -1}
		\put (59,20) {\color{red} 0}
		\put (79.5,14) {\color{cpurple} -1}
		\put (84, 11.5) {\color{lygreen} -1}
		\put (94,20) {\color{red} 0}
		
		\put (5.5, 18) {$K$}
		\put (39, 18) {$K$}
		\put (74.5, 18) {$K$}
		\put (84, 5) {$n-4$}
	\end{overpic}
	\caption{Handle-sliding}\label{fig: surgery diagram 3}
\end{figure}
Since the red curve denotes the surgery in the original $3$-manifold, we can further handle-slide the knot over it as in Figure \ref{fig: surgery diagram 3}. This corresponds to an isotopy of the whole $4$-manifold. The resulting curve is an unknot linking with the green and the purple curves. We put a star on the black curve to denote the position of the suture in the construction of $KHI$ (or the earring in the construction of $\ina$). Then the black curve with the star indicates a surface (indeed an annulus) with an arc inside the cobordism.

To visualize the surface, we do extra Kirby moves as in Figure \ref{fig: surgery diagram 4}, where we first handle-slide the green curve over the purple curve, and then perform the isotopies as indicated by the orange arrows. 

\begin{figure}[htbp]
\centering
	\begin{overpic}[width=0.8\textwidth]{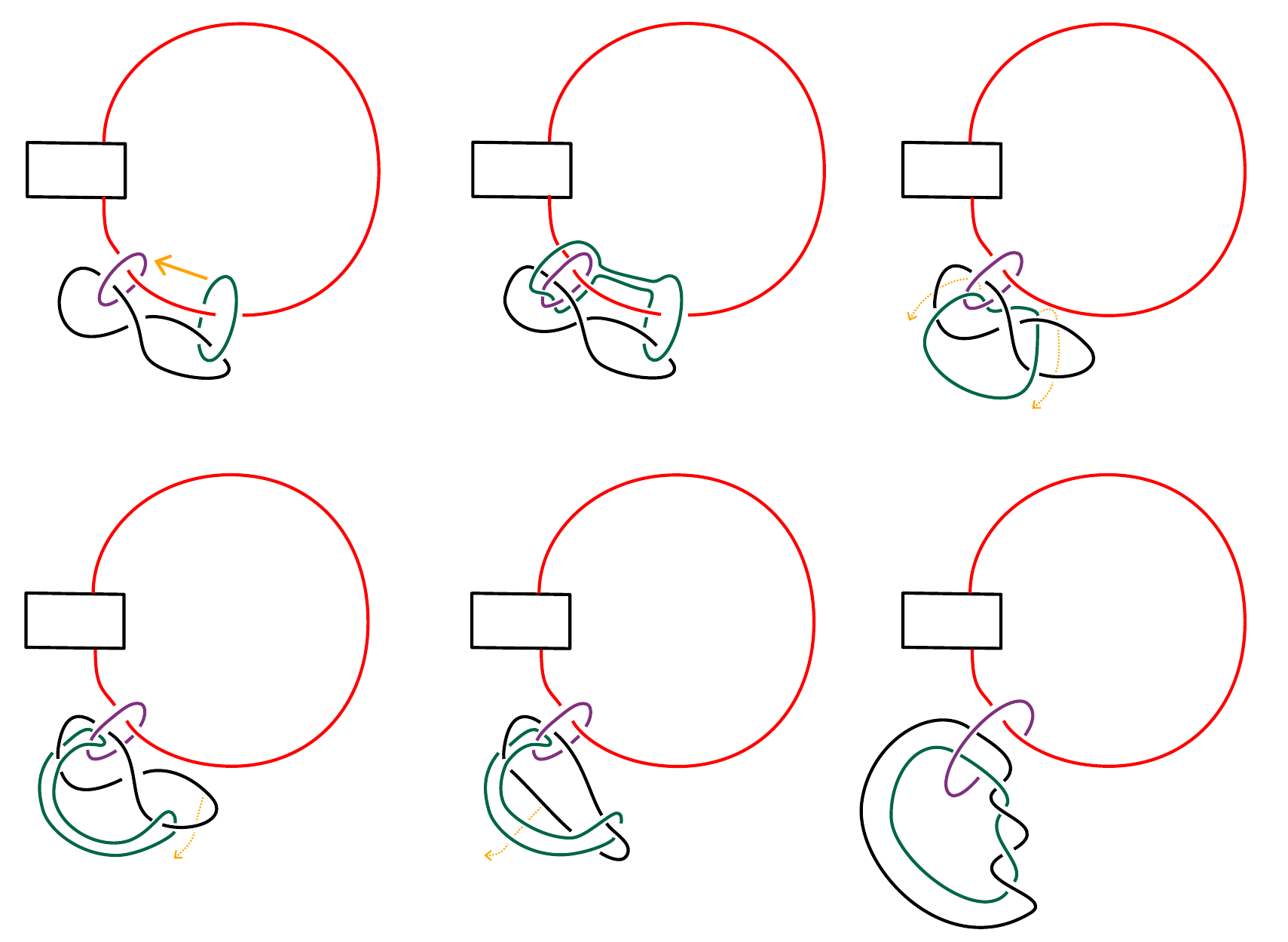}
		\put(5,24) {$K$}
		\put(39.5,24) {$K$}
		\put(73,24) {$K$}
		\put(5,59) {$K$}
		\put(39.5,59) {$K$}
		\put(73,59) {$K$}
		
		\put(10,19) {\color{cpurple}$-1$}
		\put(44,19) {\color{cpurple}$-1$}
		\put(78,19) {\color{cpurple}$-1$}
		\put(10,54) {\color{cpurple}$-1$}
		\put(78,54) {\color{cpurple}$-1$}
		
		\put(16,52.5) {\color{lygreen}$-1$}
		\put(50,52.5) {\color{lygreen}$-2$}
		\put(75,40) {\color{lygreen}$-2$}
		\put(7,8.5) {\color{lygreen}$-2$}
		\put(34,8.5) {\color{lygreen}$-2$}
		\put(70,8.5) {\color{lygreen}$-2$}
		
		\put(25,20) {\color{red} 0}
		\put(60,20) {\color{red} 0}
		\put(94,20) {\color{red} 0}
		\put(26,55) {\color{red} 0}
		\put(61,55) {\color{red} 0}
		\put(94,55) {\color{red} 0}
	\end{overpic}
	\caption{Kirby moves}\label{fig: surgery diagram 4}
\end{figure}

\begin{figure}[htbp]
\centering
\begin{minipage}{0.35\textwidth}
	\begin{overpic}[width=0.98\textwidth]{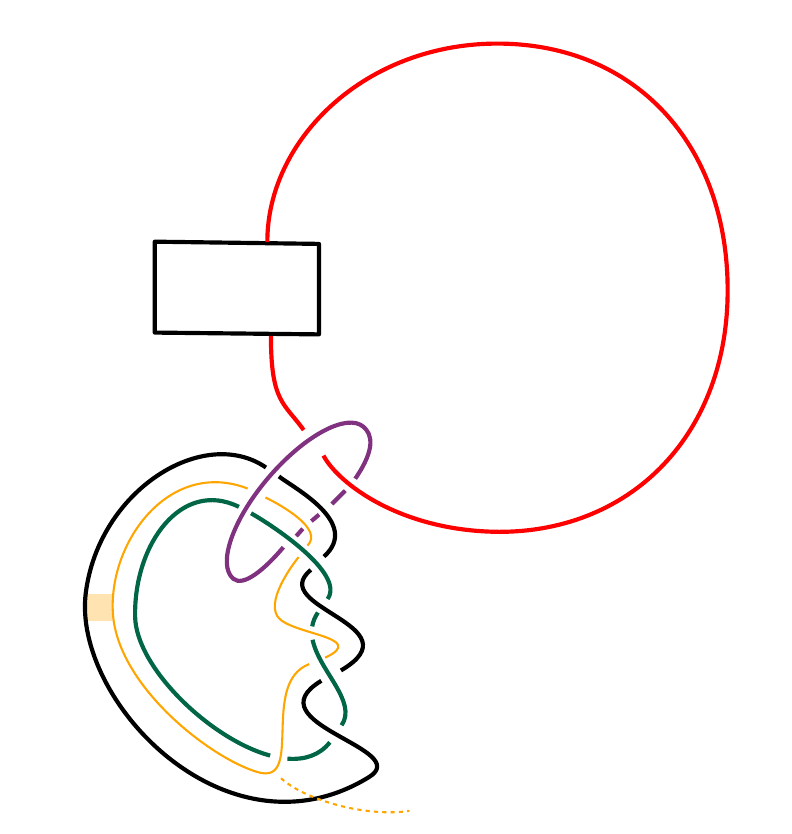}
		\put(25,65) {$K$}
		
		\put(78,50) {\color{red} $0$}
		
		\put(50,0.5) {\color{orange} $D_G$}
		\put(2,25) {\color{orange} $B_1$}
		
		\put(19,25) {\color{lygreen} $-2$}
		
		\put (45,50) {\color{cpurple} $-1$}
	\end{overpic}

    \end{minipage}
\begin{minipage}{0.64\textwidth}
\begin{overpic}[width=0.98\textwidth]{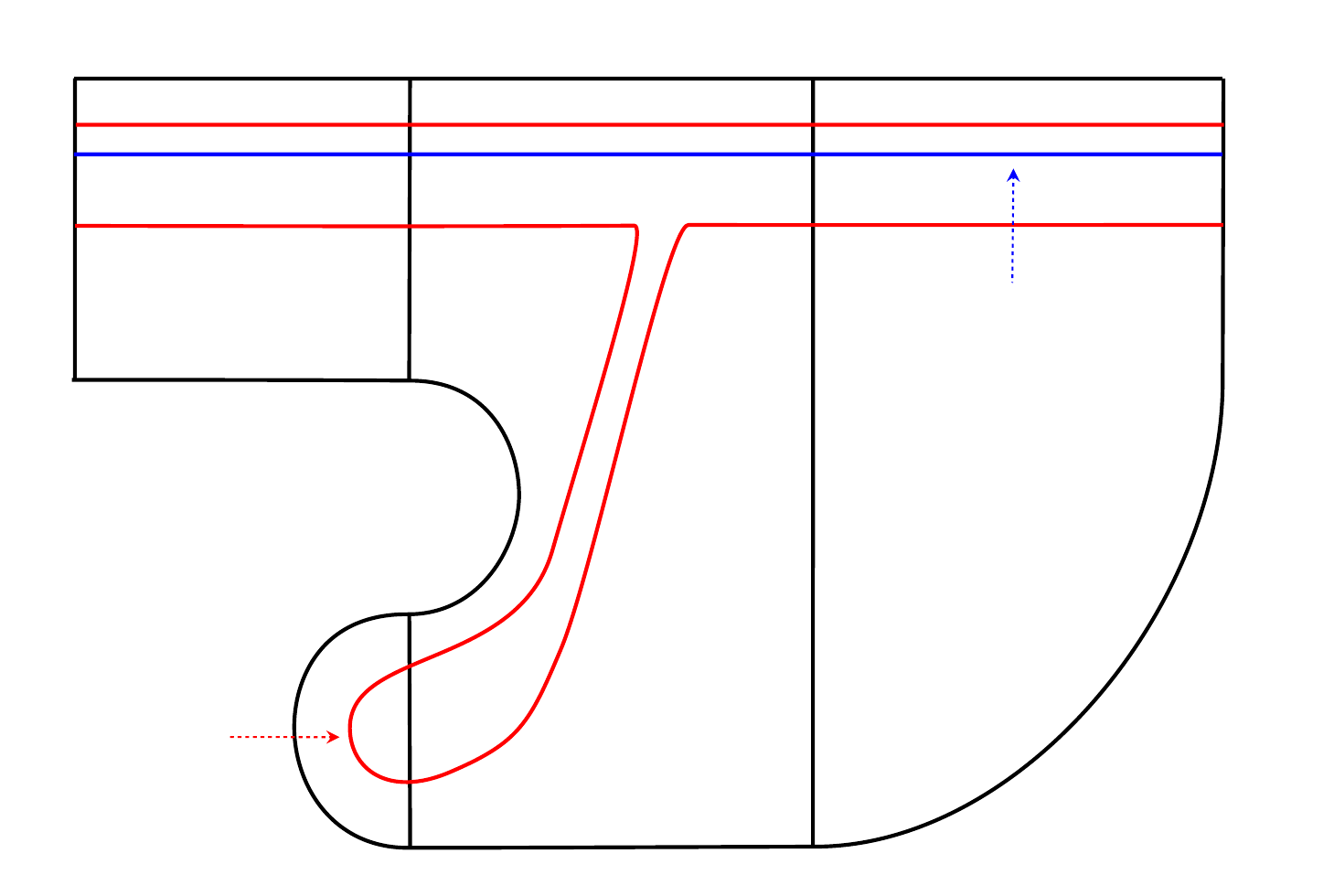}
		\put (2, 63) {$Y_0(K)$}
		\put (22, 63) {$Y_0(K)\sqcup \mathbb{RP}^3$}
		\put (54, 63) {$Y_0(K)\# \mathbb{RP}^3$}
		\put (89, 63) {$Y_2(K)$}
		
		\put (10, 44) {\color{red}$\widetilde{A} = A \# S_0$}
		\put (41, 44) {\color{red}$B_1$}
		\put (11, 11) {\color{red}$D_G$}
		\put (63, 42) {\color{blue} suture or earring}
	\end{overpic}
\end{minipage}
	\caption{Two visualizations of the surface $\widetilde{A}$}\label{fig: surgery diagram 5}
\end{figure}
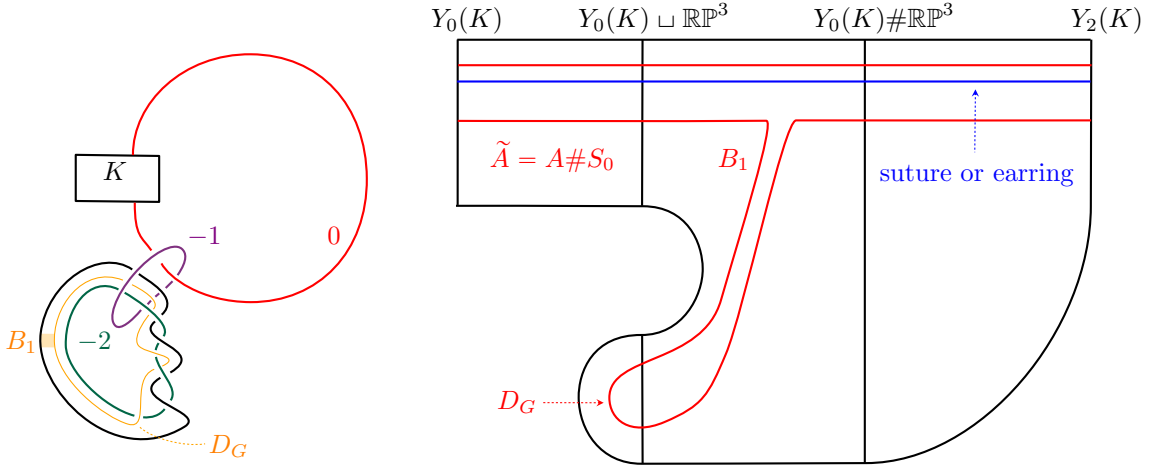
		
We use the last diagram in Figure \ref{fig: surgery diagram 4} to describe the cobordism associated with $ G_\be \circ F_\al$ as follows. 

First, we consider the cobordism associated with the green curve. Since it is disjoint from the red curve and has framing $-2$, we know it is a cobordism \[W_G:Y_0(K)\to Y_0(K)\# \mathbb{RP}^3\] obtained from the boundary connected sum of $Y_0(K)\times I$ and $\mathbb{D}({-2})$, the disk bundle over a $2$-sphere with Euler number $-2$.

Second, we consider the cobordism $W_P$ associated with the purple curve. It is the usual surgery cobordism with framing $-1$.

Alternatively, from the last diagram in Figure \ref{fig: surgery diagram 3}, we know the cobordism is just the composition of two surgery cobordisms with framings $-1$. Note that if we do not care about the embedded surface, then the Kirby diagram is just the same as Figure \ref{fig: W}.

\subsection{Embedded surface}\label{subsec: embedded surfaces}

Then we consider the surface associated with the black curve, which is depicted in the left diagram in Figure \ref{fig: surgery diagram 5}. The black curve times the interval gives an annulus $A_0$ inside the cobordism. Before doing any surgery, the black curve is an unknot that does not link with any curve. After doing surgeries along the green curve, the black curve is isotopic to an unknot, but the isotopy will change the surface. To describe the surface, we introduce another orange curve around the green curve to indicate a core disk $D_G$. The band sum along $B_1$ gives an unknot disjoint from other curves. Let $A_1$ be the product annulus of this unknot inside $W_P$. Define \begin{equation}\label{eq: defn of A tilde}
    \widetilde{A}\deq A_0\cup D_G\cup B_1\cup A_1.
\end{equation}Equivalently, let $S_0$ be the zero section of $\mathbb{D}({-2})$. Then $\widetilde{A}$ is the connected sum of $S_0$ and the  product annulus $A$ of an unknot away from all curves, as depicted as the red surface (with a blue arc) in the right diagram of Figure \ref{fig: surgery diagram 5}.

\section{The commutative diagram}\label{sec: The commutative diagram}

This section is devoted to the proof of Proposition \ref{prop: commutative}.
\subsection{Stretching argument}
In this subsection, we consider the stretching argument along the boundary of the neighborhood of an embedded $2$-sphere of self-intersection $-2$ in the cobordism schematically described in \Cref{fig: surgery diagram 5}. In preparation for this, we first need some preliminary moduli computations.

Let $Z = \mathbb{D}({-2})$, the disk bundle over a $2$-sphere with Euler number $-2$. Let $U \subset \partial Z \cong \rp$ be a local unknot, $D \subset Z$ be the push-off into $Z$ of the disk that $U$ bounds, and let $S_0$ denote the zero section lying in $Z$.


Let $P$ denote the trivial $SO(3)$ bundle over $Z$. A direct computation shows that there are exactly two elements in\[\chi (\rp, U)=\{\rho: \pi_1(\rp\backslash U)\to SU(2)\mid \operatorname{tr}(m_U)=0\text{ for the meridian }m_U\text{ of }U\}/SU(2),\] where $SU(2)$ acts by conjugation.

Denote these two generators by $\theta_\pm$ that are distinguished by the image of the generator of $\pi_1 (\rp)$: $\theta_+$ corresponds to image $\id\in SU(2)$ and $\theta_-$ corresponds to image $-\id\in SU(2)$.

\bprop \label{prop: moduli-replacement}
The moduli space of minimal energy reducible ASD connections, $\mcm^{red}_{min} (Z, S_0\# D, P)$, consists of points with index $-1$ that correspond one-to-one with $\alpha \in \chi(\rp, U)$.
\eprop

\bpf
We will drop $P$ from the notation in this proof for readability with the understanding that all bundles have trivial $w_2$.

Let $S$ denote a $2$-sphere in a $4$-ball in the neighborhood of $\partial Z$. Then \[(Z, S_0 \# S) \cong (Z, S_0 \#D) \cup (\partial Z \times I, \overline{D}),\]where $\overline{D}$ denotes the reverse of $D$. Since $(Z, S_0\# S) \cong (Z, S_0)$, we can appeal to the proof of \cite[Proposition 6.1]{bhat2023newtriangle} to conclude that the moduli space of minimal energy $\mcm^{red}_{min} (Z, S_0\# S; \beta)$ consists of a single point for each $\beta \in \chi (\rp)$, where $\chi(\rp)$ is the $SU(2)$ character variety of $\rp$. In addition, we have
\begin{equation} \label{eq: neck-stretch}
\mcm^{red}_{min} (Z, S_0\# S) \cong \mcm^{red}_{min} (Z, S_0\# D) \times_{\chi (\rp, U)} \mcm^{red}_{min} (\partial Z \times I, \overline{D}).
\end{equation}
We remark that obstruction bundles are all trivial since the map $H^2_c(-; \R) \to H^2(-; \R)$ are all zero so that the standard gluing arguments apply unchanged to conclude the above statement.

Since the left-hand side consists of reducibles with $U(1)$ stabilizers and $\chi (\rp, U)$ also consists of $U(1)$ reducibles, we see that every moduli space in (\ref{eq: neck-stretch}) is stabilized by $U(1)$ which is identified by the inclusion of $(\rp, U)$ into the relevant factors; in particular, there is no gluing parameter.  

Note that any $\alpha \in \chi (\rp, U)$ extends to a unique flat connection $A_\alpha$ on $(\partial Z \times I, \overline{D})$ with $\ind (A_\alpha) = -1$ and the holonomies along the generator of $\pi_1 (\rp)$ on both ends are equal. 

Noting that the energy $\kappa (A)$ depends only on $\ad (A)$, and that all elements of $\chi (\rp, U)$ and $\chi (\rp)$ induce the same representation on the adjoint bundle, we conclude that $\kappa (A) \geq 1/4$ if $\kappa(A) \neq 0$, where $[A] \in \mcm^{red} (\partial Z \times I, \overline{D})$, for any such $A$, $\ad (A)$ is obtained by adding monopole and/or instanton charges to $\ad (A_\alpha)$. The proof of \cite[Proposition 6.1]{bhat2023newtriangle} implies that $\kappa = 1/8$ for any connection in $\mcm^{red}_{min} (Z, S_0\# S)$ with minimal energy. Thus, $\mcm^{red}_{min} (\partial Z \times I, \overline{D})$ consists of flat connections, i.e.\, $A_\alpha$ for every $\alpha \in \chi (\rp, U)$.

The conclusion above can be rephrased as $\mcm^{red}_{min} (\partial Z\times I, \overline{D}) \to \chi(\rp, U)$ is a homeomorphism, so we can replace the fiber-product in (\ref{eq: neck-stretch}) by $\mcm^{red}_{min} (Z, S_0\# D)$, which concludes the proof with a final appeal to \cite[\S 6]{bhat2023newtriangle} and the linear index gluing formula.
\epf

\bprop \label{prop: equal-upto-sign}
The composition map $\phi=G_\beta\circ F_\alpha$ from Proposition \ref{prop: commutative} is equal to $f_0^+\pm  f_0^-$ up to a global sign, where the sign ambiguity before $f_0^-$ (the relative sign) does not depend on the knot.
\eprop

\bpf
To begin with, we will ignore the orientation of all moduli spaces and only comment on this at the end of the proof. 

To compare these cobordism maps, we first note that the cobordisms used in the maps $f_0^\pm$ are obtained by first removing a neighborhood of $S_0$ defined in \S \ref{subsec: embedded surfaces} and then gluing in a $(\rp \times I, D)$, where $D$ is a capping off disk for the local unknot $U$ on the boundary of the piece being removed.

Neck stretching along the boundary of the neighborhood of $S_0$, we break the connection into two connections on the resulting 4-manifold pieces. Calling the connection on the neighborhood of $S_0$ as $A_Z$ and on the rest of the manifold as $A$, we have by linear excision that 
\[
0 = \ind (A) + \ind (A_Z) + h^0(\alpha) + h^1(\alpha),
\]
where $\alpha \in \chi (\rp, U)$; linear excision here refers to the form of index formula as described in \cite[\S 4.3, eq (11)]{scaduto2015instanton}. Note that $h^0(\alpha) = 1$, $h^1(\alpha) = 0$, and $\ind A \geq 0$---the last one is due to the bundle data for $A$ being non-integral and assuming appropriate perturbations have been chosen to ensure all irreducible moduli are transverse. Thus, $\ind (A_Z) \leq -1$. By Proposition \ref{prop: moduli-replacement}, we have $\ind (A_Z) \geq -1$, forcing $\ind A = 0$ and $\ind (A_Z) = -1$.

To complete the proof, what we need to show is that 
\[
\mcm = \left \{ [A] \in \mcm (\rp \times I, D) \, \lvert \, \ind (A) = -1 \right \},
\]
consists of reducibles and that $\mcm \to \chi (\rp)$ and $\mcm \to \chi(\rp, U)$ are homeomorphisms. That latter is equivalent to showing the maps are bijections.

Since we can ensure that the irreducible connections in $\mcm$ are cut out transversely with $\ind \geq 0$, all elements of $\mcm$ must be reducible. Thus, we are concerned with the reducible solutions on $(\rp \times I, D)$. We first note that for any $\alpha \in \chi (\rp, U)$, there is a unique flat connection on $(\rp \times I, D)$: say $A_\alpha$. Note that $A_\alpha$ are central connections and so $\ad(A_\alpha)$ is trivial. Since $H^2_{c}(\rp \times I, D; \mathbb{R}) = 0$ and $H^1_{c}(\rp \times I, D; \mathbb{R}) = 0$, $A_\alpha$ is unobstructed with $\ind (A_\alpha) = -1$. Since they are flat, they are of minimal energy, and any higher-energy solution to the ASD equations will have a larger index as long as the endpoints are fixed. Lastly, we need to show that the representations at the boundary are the same as for $A_\alpha$. We will show that any reducible ASD connection on $(\rp \times I, D)$ must be flat. Suppose $A$ is such a connection then $\frac{i}{2\pi} F_A$ is a $L^2$ $i\mathbb{R}$-valued 2-form on $\rp \times I$. Since the map $H^2_c (\rp \times \R; \R) \to H^2 (\rp \times \R; \R)$ is zero, we must have that $A$ is flat.

We finally comment on the orientation of moduli spaces, which is crucial in making sense of the statement of this lemma. In \cite{bhat2023newtriangle}, we use the orientations induced by almost complex structures to define $f_0^\pm$---these are essentially the ones in \cite{kronheimer2011khovanov} with a few minor modifications as detailed in \cite[\S 3.7--3.8]{bhat2023newtriangle}. A key point to note is that we \emph{cannot} ensure that the almost complex structures used there can be used to orient the moduli spaces appearing here simultaneously. Instead of attempting to compare them, we utilize the homological orientations used in \cite[\S 3]{kronheimer2011khovanov}. For cobordisms, we need to choose an $I$-orientation as defined in \cite[Definition 3.9]{kronheimer2011khovanov}---this requires us to make a choice of a `basepoint' connection $A$ along with a choice of trivialisation of the $\det (\mathcal{D}_A)$---which behaves well under neck-stretching. Thus, assuming we chose a basepoint connection on the cobordism defining $\phi$ and on $(\rp \times I, D)$, we would either have $\phi = \pm (f_0^+ + f_0^-)$ or $\phi = \pm (f_0^+ - f_0^-)$; that is, the only ambiguity is the global sign---due to the two different orientation conventions---but not the relative sign, as cobordism and bundle data underlying $f_0^\pm$ are identical.
\epf

\subsection{Fixing the sign}
In this subsection, we fix the relative sign in Proposition \ref{prop: equal-upto-sign} by a direct computation for the right-handed trefoil. We adopt the notation in \S \ref{sec: Triangles in singular instanton homology}. In particular, for a fixed framed knot $K\subset Y$, let $W^n_{n+2}:Y_n(K)\to Y_{n+2}(K)$ be the composition of two surgery cobordisms $W^m_{m+1}:Y_{m}(K)\to Y_{m+1}(K)$ for $m=n,n+1$ and let $\Omega\subset W^n_{n+2}$ be the embedded $2$-sphere with self-intersection $-2$.
\begin{proposition}\label{prop: choose sign}
	If there exists $\epsilon\in\{\pm 1\}$, independent of the choice of the framed knot $K\subset Y$, such that the following exact triangle holds:
	\begin{equation}\label{eq: ET from sutures}
		\xymatrix{
			I^{\sharp}(Y_0(K);\mathbb{C})\ar[rr]^{I^{\sharp}(W^0_2) + \epsilon\cdot I^{\sharp}_\Omega(W^0_2)}&&I^{\sharp}(Y_{2}(K);\mathbb{C})\ar[dl]\\
			&H_*(\cone(d_1^+ + d_1^-:KHI(Y,K)\to KHI(Y,K)))\ar[ul]&
		}
	\end{equation}
	Then $\epsilon$ must be $-1$.
\end{proposition}
\bpf
We can compute the sign for a special knot. Let $T\subset S^3$ be the right-handed trefoil. We consider its $1$-framing under the assumption, with respect to the Seifert framing so that $W^0_2$ should be replaced by $W^1_3:S^3_1(K)\to S^3_3(K)$. We omit the coefficients $\mathbb{C}$ in all instanton homologies. 

From \cite[Theorems 1.1 and 1.13]{baldwin2020concordance}, we know that
\begin{equation}\label{eq: dim of I-sharp for trefoil}
	\dim I^{\sharp}(S^3_n(T)) = n \text{ for } n\in\posi
\end{equation}
Also, from \cite[Theorem B.2]{LY2025torsion}, we know that $\dim I^{\sharp}(S^3,T) = 4$, which, together with Propositions \ref{prop: original triangle} and \ref{prop: interpretation of f_0}, implies that 
\begin{equation}\label{eq: equal maps}
    I^{\sharp}(W^1_{3}) - I^{\sharp}_\Omega(W^1_{3}) = 0.
\end{equation}
Hence \begin{equation}\label{eq: equal maps 2}
     I^{\sharp}(W^1_{3}) + I^{\sharp}_\Omega(W^1_{3})=2\cdot I^{\sharp}(W^1_{3})
\end{equation}

From \cite[Theorem 2.1]{scaduto2015instanton} and \cite[Proposition 3.9]{LY2025dimension}, we have the following two exact triangles.

\[
	\xymatrix@C=2cm{
		I^{\sharp}(S^3_1(T))\ar[r]^{I^{\sharp}(W^1_2)}&I^{\sharp}(S^3_2(T))\ar[r]^{I^{\sharp}(W^2_3)}\ar@<-1ex>[d]&I^{\sharp}(S^3_3(T))\ar[dl]\\
		&I^{\sharp}(S^3)\ar[lu]\ar@<-1ex>[u]&
	}
\]
Hence (\ref{eq: dim of I-sharp for trefoil}) implies that $I^{\sharp}(W^1_2)$ and $I^{\sharp}(W^2_3)$ are both injective and so is $I^{\sharp}(W^1_3)$.


From (\ref{eq: equal maps}), (\ref{eq: equal maps 2}), and the injectivity of $I^{\sharp}(W^1_3)$, to conclude $\epsilon = -1$ in (\ref{eq: ET from sutures}), it suffices to show that for the right-handed trefoil, we have
\begin{equation}\label{eq: computation of cone}
    \dim H_*(\cone(d_1^+ + d_1^-)) = 4
\end{equation}
as well. Since the right-handed trefoil is an instanton L-space knot, the description for $d_1^\pm$ can be found in \cite[Theorem 5.11]{LY2021large} as in the proof of Proposition \ref{prop: l space knot computing}. Explicitly, we have
\[
KHI(S^3, T, i) \cong \begin{cases}
	\mathbb{C} & i\in\{-1,0,1\}\\
	0 & {\rm otherwise}
\end{cases}
\]
and $\im (d_1^+) = \im(d_1^-) = KHI(S^3, T, 0)$. Thus, the computation in (\ref{eq: computation of cone}) follows.
\epf

\bpf[Proof of Proposition \ref{prop: commutative}]
By Proposition \ref{prop: -1 triangle} and Corollary \ref{cor: composition of H maps}, there is a variant of the octahedral diagram (\ref{eq: octahedral used 1})
\begin{equation}\label{eq: octahedral used different}
\xymatrix{
	&H_*(\cone(c_+ d_1^++c_- d_1^-))\ar@{..>}[dr]\ar[ddl]&\\
	\ish (Y_2(K);\comp) \ar@{..>}[ur]\ar[d]&&\ish (Y_0(K);\comp)\ar@{..>}[ll]^{G_\be\circ F_\al}\ar[ddl]^>>>>>>>>>>>>{F_\al}\\
 KHI(Y,K)\ar[dr]_{H_\beta}\ar[rr]^{c_+ d_1^++c_- d_1^-}&&KHI(Y,K)\ar[u]\ar[uul]\\
	&KHI(Y_1(K),\widetilde{K}_1)\ar[uul]^<<<<<<<<<<<<{G_\be}\ar[ur]_{H_\al}&
	}
    \end{equation}
    which implies the dotted exact triangle.
From Propositions \ref{prop: equal-upto-sign} and \ref{prop: interpretation of f_0}, we know that\[G_\be\circ F_\al=\pm (f_0^++\epsilon\cdot f_0^-)=\pm (I^{\sharp}(W) + \epsilon\cdot I^{\sharp}_\Omega(W))\]for some $\epsilon\in\{\pm 1\}$ that is independent of $(Y,K)$, where the global sign does not matter since it will not affect the kernel and the image. When $K$ is rationally null-homologous, we know $d_1^\pm$ have different grading shifts and then\[\cone(c_+ d_1^++c_- d_1^-)\cong \cone(d_1^++d_1^-).\]By Proposition \ref{prop: choose sign}, we know that $\epsilon=-1$, which concludes the proof.
\epf

\bibliographystyle{alpha}

\end{document}